\documentclass[10pt,final]{amsart}
    \usepackage{ifdraft}
\usepackage[maxbibnames=99,backend=bibtex, style=numeric]{biblatex}
\AtEveryBibitem{
    \clearfield{issn}
  \clearfield{isbn}
  \clearlist{location}
  \clearfield{month}
  \clearfield{day}
  \clearfield{series}
  \clearfield{pages}
  \clearlist{language}
  \clearlist{translator}
\renewbibmacro{in:}{}
}
\usepackage{amssymb,euscript}
    \usepackage{graphicx}
    \usepackage[english]{babel}
    \usepackage{graphicx}
    \usepackage{framed}
    \usepackage[normalem]{ulem}
    \usepackage{amsmath}
    \usepackage{dynkin-diagrams}
    \usepackage{amsthm}
    \usepackage{amssymb}
    \usepackage{hyperref} 
    \usepackage[capitalize]{cleveref}
    \crefformat{equation}{(#2#1#3)}
\crefrangeformat{equation}{(#3#1#4--#5#2#6)}
\crefformat{Corollary}{Corollary~#2#1#3}
\crefmultiformat{Corollary}{Corollaries~#2#1#3}{ and~#2#1#3}{, #2#1#3}{ and~#2#1#3}
\crefformat{Remark}{Remark~#2#1#3}
\crefmultiformat{Remark}{Remarks~#2#1#3}{ and~#2#1#3}{, #2#1#3}{ and~#2#1#3}
    \usepackage{comment}
    \usepackage{tikz}
    \usetikzlibrary{matrix,calc,decorations.pathreplacing}
    \usepackage{mathtools}
    \usepackage{tikz-cd}
    \usepackage{amsfonts}
    \usepackage{enumerate}
    \usepackage{quiver}
\usepackage{xcolor}
    \usepackage[utf8]{inputenc}
    \usepackage[top=1 in,bottom=1in, left=1 in, right=1 in]{geometry}
       \usepackage{stmaryrd}
       
    \title{Functoriality of Coulomb Branches}
    \author{Tom Gannon and Ben Webster}

    \usepackage[colorinlistoftodos]{todonotes}
    \usepackage[all]{xy}
    \usepackage{mathrsfs}

    \newcommand{\C}{\mathbb{C}} 
    \newcommand{\Z}{\mathbb{Z}} 
    
    \newcommand{\LSL}{\mathfrak{sl}}
    \newcommand{\SL}{\operatorname{SL}}
     \newcommand{\GL}{\operatorname{GL}}   
    \newcommand{\PGL}{\operatorname{PGL}}

    \newcommand{\LGL}{\mathfrak{gl}}
    
    \newcommand{\LG}{\mathfrak{g}}
    
    \newcommand{\LGd}{\mathfrak{g}^*}
    
    \newcommand{\LtG}{\mathfrak{\tilde{g}}}
    \newcommand{\LtGc}{\mathfrak{\tilde{g}}^{\circ}}

    \newcommand{\LH}{\mathfrak{h}}
    \newcommand{\LT}{\mathfrak{t}}
    
    \newcommand{\LF}{\mathfrak{f}}
    \newcommand{\LtT}{\mathfrak{\tilde t}}
    
    \newcommand{\LtTc}{\mathfrak{\tilde t}^{\circ}}

    \newcommand{\Kost}[1]{\mathbb{K}_{#1}}

    \newcommand{\End}{\operatorname{End}} 
    \newcommand{\Hom}{\operatorname{Hom}}

    \renewcommand{\Hom}{\operatorname{Hom}} 
    \renewcommand{\End}{\operatorname{End}} 
    \newcommand{\QCoh}{\operatorname{QCoh}}

    \renewcommand{\\}{\backslash}
    \renewcommand{\subset}{\subseteq}

    \theoremstyle{definition}
    \newtheorem{Theorem}{Theorem}[section]
    \newtheorem{Conjecture}[Theorem]{Conjecture}
    
    \newtheorem{Corollary}[Theorem]{Corollary}
    \newtheorem{Definition}[Theorem]{Definition}
    \newtheorem{Proposition}[Theorem]{Proposition}
    
    \newtheorem{Remark}[Theorem]{Remark}
    
    \newtheorem{Lemma}[Theorem]{Lemma}

    \newcommand{\Gm}{\mathbb{G}_m}
     \newcommand{\tG}{\tilde{G}}
     \newcommand{\tT}{\tilde{T}}
\newcommand{\birat}{\varpi}
\numberwithin{equation}{section}

\newcommand{\Spec}{\mathrm{Spec}}   
\newcommand{\tM}{\widetilde{\mathcal{M}}}

\newcommand{\tH}{\tilde{H}}
\newcommand{\LtH}{\tilde{\mathfrak{h}}}
\newcommand{\tcH}{{\mathfrak{c}}_{\tH}}
\newcommand{\tcG}{{\mathfrak{c}}_{\tG}}

\newcommand{\AGN}{\mathcal{A}_{G, \mathbf{N}}}
\newcommand{\sphericalHeckeCatForTHISGROUP}[1]{\overline{\mathcal{H}_{#1}}}
\newcommand{\tc}[1]{\tilde{\mathfrak{c}}_{#1}}
\newcommand{\fc}[1]{{\mathfrak{c}}_{#1}}
\newcommand{\tcGc}{\tcG^{\circ}}
\newcommand{\tcHc}{\tcH^{\circ}}

\newcommand{\bfN}{\mathbf{N}}
\newcommand{\tA}{\tilde{\mathcal{A}}}
\newcommand{\cH}{\mathfrak{c}_H}
\newcommand{\cG}{\mathfrak{c}_G}
\newcommand{\mapFromHToG}{\varsigma}
\newcommand{\mapFromBaseChangedCoulombBranchforGtoCoulombBranchforH}{\underline{\mapFromHToG}}
\newcommand{\M}{\mathcal{M}}
\newcommand{\Gr}{\mathrm{Gr}}
\newcommand{\semisimpleComplexesOnGr}{\mathcal{I}}
\newcommand{\LGc}{\Check{\LG}}
\newcommand{\LGcd}{\Check{\LG}^*}

\newcommand{\Sym}{\mathrm{Sym}}

\newcommand{\aff}{\mathbb{C}}

\begin{document}

 \newcommand{\Ben}[1]{%
     \ifdraft{\par\noindent
     {\color{red}\textbf{BW:} #1}%
     \par}
 }
 \newcommand{\Tom}[1]{
     \ifdraft{\par\noindent
          {\color{blue}
              \textbf{TG:} #1
          }}
 \par    
 }

\maketitle
\begin{abstract}
We prove that the affine closure of the cotangent bundle of the parabolic base affine space for $\mathrm{GL}_n$ or $\mathrm{SL}_n$ is a Coulomb branch, which confirms a conjecture of Bourget-Dancer-Grimminger-Hanany-Zhong. In particular, we show that the algebra of functions on the cotangent bundle of the parabolic base affine space of $\mathrm{GL}_n$ or $\mathrm{SL}_n$ is finitely generated. 

We prove this by showing that, if we are given a map $H \to G$ of complex reductive groups and a representation of $G$ satisfying an assumption we call gluable, then the Coulomb branch for the induced representation of $H$ is obtained from the corresponding Coulomb branch for $G$ by a certain Hamiltonian reduction procedure. In particular, we show that the Coulomb branch associated to any quiver with no loops can be obtained from Coulomb branches associated to quivers with exactly two vertices using this procedure.
\end{abstract}

\section{Introduction}
The Coulomb branch $\M_C(G, \bfN)$ associated to a (connected complex) reductive group $G$ and a representation $\bfN$ was defined precisely by Braverman-Finkelberg-Nakajima in \cite{BravermanFinkelbergNakajimaRingObjectsIntheEquivariantDerivedSatakeCategoryArisingFromCoulombBranches}. Our first main result, which we state precisely in \cref{Intro Statement of Functoriality Theorem}, informally states that, given a map of reductive groups $\mapFromHToG: H \to G$, the variety $\M_C(H, \bfN)$ can be obtained from $\M_C(G, \bfN)$ in a systematic way assuming that $\mapFromHToG$ satisfies a certain assumption which we call \textit{gluable}, defined precisely in \cref{def:gluable}. We use this to derive our other two main results: 

\begin{enumerate}
    \item A proof of the fact that the Coulomb branch of any \textit{quiver gauge theory} corresponding to a quiver with no loops is entirely determined by Coulomb branches of quivers with exactly two vertices, discussed in more detail in \cref{Dismemberments of Quiver Gauge Theories Subsection} and
    \item A proof of a conjecture of Bourget-Dancer-Grimminger-Hanany-Zhong \cite{BourgetDancerGrimmingerHananyZhongPartialImplosionsandQuivers} that certain varieties which appear naturally in the context of representation theory and holomorphic symplectic geometry known as \textit{partial implosions} for a reductive group $G$ are Coulomb branches when $G = \GL_n$ or $G = \SL_n$, see \cref{Corollary for Partial Implosion}. 
\end{enumerate}

\subsection{Functoriality of Coulomb branches}\label{Intro Functoriality of Coulomb Branches Subsection} We now state our first main result precisely. By construction (see, for example, \cite[Section 3(vi)]{BravermanFinkelbergNakajimaTowardsaMathematicalDefinitionOfCoulombBranchesOf3dNEquals4GaugeTheoriesII}), the Coulomb branch $\M_C(G, \bfN)$ admits a canonical map to the affine scheme $\LG\sslash G$ whose ring of functions is the $G$-equivariant functions on $\LG := \mathrm{Lie}(G)$ and, furthermore, under this map, $\M_C(G, 0)$ acquires the structure of a group scheme over $\LG\sslash G$, see, for example, \cite[Section 2.6]{YunZhuIntegralHomologyofLoopGroupsviaLanglandsDualGroups}. We also recall a result of Teleman \cite{TelemanCoulombBranchesforQuaternionicRepresentations}, which we reprove in \cref{Action for General Coulomb Branches}, which gives an action of the group $\mathcal{M}_C(G, 0)$ on $\mathcal{M}_C(G, \bfN)$ as schemes over $\LG\sslash G$. 

As above, we fix a map $\mapFromHToG: H \to G$ of reductive groups. The functoriality of equivariant homology gives a map \begin{equation}\label{Group Map for Change of Groups Coulomb Branch}\M_C(G, 0) \times_{\LG\sslash G} \LH\sslash H \to \mathcal{M}_C(H, 0)\end{equation} of group schemes over $\LH\sslash H$. In particular, there is a right action of $\M_C(G, 0)$ on $\M_C(H, 0)$ and so we may define the balanced product \[\M_C(H, 0) \times_{\LG\sslash G}^{\M_C(G, 0)} \M_C(G, \bfN)\] as the affine scheme whose ring of functions is the subring of functions on $\M_C(H, 0) \times_{\LG\sslash G} \M_C(G, \bfN)$ such that the pullbacks by the maps \[\M_C(H, 0) \times_{\LG\sslash G} \M_C(G, 0) \times_{\LG\sslash G} \M_C(G, \bfN) \rightrightarrows \M_C(H, 0) \times_{\LG\sslash G} \M_C(G, \bfN)\] induced by the above actions agree. Our first main result states that this balanced product gives the Coulomb branch for the restriction of the representation $\bfN$ to $H$ under a mild technical hypothesis on the map $\mapFromHToG$ which we call \textit{gluable} (\cref{def:gluable}):

\begin{Theorem}\label{Intro Statement of Functoriality Theorem}
If the map $\mapFromHToG$ is gluable for the representation $\bfN$, there is a canonical $\M_C(H, 0)$-equivariant isomorphism \begin{equation}\label{Intro Isomorphism In First Theorem}\M_C(H, 0) \times_{\LG\sslash G}^{\M_C(G, 0)} \M_C(G, \bfN) \xrightarrow{\sim} \M_C(H, \bfN)\end{equation} of affine varieties over $\LH\sslash H$.
\end{Theorem}
\begin{Remark}
The obvious analogues of \cref{Intro Statement of Functoriality Theorem,gluing for quiver case} hold for the K-theoretic Coulomb branch (denoted $\EuScript{C}_4(G; \bfN)$ in \cite{TelemanTheRoleofCoulombBranchesin2DTheory}); that is,
\begin{equation}\label{Intro K-theory Isomorphism}\EuScript{C}_4(H, 0) \times_{G\sslash G}^{\EuScript{C}_4(G, 0)} \EuScript{C}_4(G, \bfN) \xrightarrow{\sim} \EuScript{C}_4(H, \bfN).\end{equation}
The proof is so similar that we will only include a few remarks (\cref{K-theory 1,K-theory 2}) about the necessarily changes in the definition and theorems.
\end{Remark}

Observe that the domain of the map \cref{Intro Isomorphism In First Theorem} can be obtained from the product $\M_C(H, 0) \times  \M_C(G, \bfN)$ by first restricting to a coisotropic subscheme and then quotienting by the null fibration, which is the orbits of $\M_C(G, 0)$. In this way, \cref{Intro Statement of Functoriality Theorem} can be informally restated as saying that $\M_C(H, \bfN)$ is obtained from $\M_C(H, 0) \times  \M_C(G, \bfN)$ by a Hamiltonian reduction procedure.  If we consider an extended Moore-Tachikawa category where the group schemes $\M_C(H, 0)$ are included as objects, we can view $\M_C(H, 0)$ as a 1-morphism $\M_C(G, 0)\to \M_C(H, 0)$ and the isomorphism \cref{Intro Isomorphism In First Theorem} can be interpreted as saying that this 1-morphism sends $\M_C(G, \bfN)$ as a Hamiltonian $\M_C(G, 0)$-space to $\M_C(H, \bfN)$. It is also instructive to compare the construction in \cref{Intro Statement of Functoriality Theorem} with that of \cite{CrooksMayrandSymplecticReduction}.

\begin{Remark}\label{We Think You Dont Need Gluability}
    It seems likely to the authors that our gluability assumption in \cref{Intro Statement of Functoriality Theorem} is not necessary; however, proving this result in general will likely require quite different techniques. We plan to develop these techniques in future work.
\end{Remark}

\subsection{Dismemberments of quiver gauge theories}\label{Dismemberments of Quiver Gauge Theories Subsection}
We now discuss a consequence of our above functoriality result, \cref{gluing for quiver case}, which informally says that the Coulomb branch associated to any quiver with no loops can be determined from Coulomb branches associated to quivers with exactly two vertices. To this end, we set the following notation, primarily following \cite{BravermanFinkelbergNakajimaQuiver}. Let $Q$ denote a quiver with finite vertex set $Q_0$ and finite edge set $Q_1$. We allow multiple edges incident to the same vertices and say that two edges in $Q_1$ are {\bf parallel} if they join the same vertices, with the same or opposite orientations. We also assume that we have fixed a dimension vector $\mathbf{n}\colon Q_0\to \Z_{\geq 0}$, where as usual, we write $n_i:=\mathbf{n}(i)$. In this case, we can consider the usual quiver gauge theory.  That is, we let $V_Q := \bigoplus_{i \in Q_0}\aff^{n_i}$, which we view as a $Q_0$-graded vector space, and let \[G(Q,\mathbf{n}):=\GL(V_Q) = \prod_{i \in Q_0}\GL(V_i)\] denote the graded endomorphisms of $V_Q$, which naturally acts on the vector space \[\bfN_Q := \bigoplus_{i \to i'}\Hom(\aff^{n_i}, \aff^{n_{i'}})\] where this sum varies over the edge set. We introduce the following definition:

 
\begin{Definition}\label{Dismemberment Definition}
	Given a quiver $(Q_0,Q_1)$, a {\it dismemberment} is a quiver $(\check{Q}_0,\check{Q}_1)$ equipped with a morphism of quivers $\gamma\colon (\check{Q}_0,\check{Q}_1)\to (Q_0,Q_1)$ which is a bijection between $\check{Q}_1\to Q_1$ and a bijection on all vertices with no edges connected to them.
\end{Definition} 
We can thus think of a dismemberment as keeping the edge set $Q_1$ fixed, but potentially carving each vertex in $Q_0$ into several separate vertices.  The dismemberment condition ensures that $\gamma$ is surjective on vertices;  thus given a choice of dimension vector $\mathbf{n}\colon Q_0\to \Z_{\geq 0}$, precomposition with $\gamma$ gives a dimension vector $\mathbf{\check{n}}$.  Under these conditions, we have a natural map of groups \begin{equation}\label{Inclusion Map by Dismemberment}\GL(V_Q)\hookrightarrow\GL(V_{\check{Q}})\end{equation} mapping $\GL(V_i)$ by the diagonal map into the product $\prod_{\gamma(j)=i}\GL(V_j)$, and this is compatible with the representations on $\bfN_{Q} = \bfN_{\check{Q}}$. Setting 
\[\M(Q,\mathbf{n})=\M(\GL(V_Q),\bfN_Q),\] we may now state the following consequence of \cref{Intro Statement of Functoriality Theorem}, which we will show in \cref{Proof of Gluing for Quiver Case Subsection}: 

\begin{Corollary}\label{gluing for quiver case}
If $Q$ has no loops and all pairs of parallel edges lift to parallel edges in $\check{Q}$, there is an isomorphism \begin{equation}\label{Gluing Isomorphisms for Quiver}\M(Q,\mathbf{n})\cong \M(G(Q,\mathbf{n}))\times_{\mathfrak{g}(\check{Q},\mathbf{n})\sslash G(\check{Q},\mathbf{n})}^{\mathcal{M}(G(\check{Q},\mathbf{\check n}))} \M(\check{Q},\mathbf{\check n})\end{equation} of affine schemes, where $\mathfrak{g}(\check{Q},\mathbf{n}) := \mathrm{Lie}(G(\check{Q},\mathbf{n}))$. 
\end{Corollary}

This result informally says that the Coulomb branch associated to any quiver $Q$ is determined by the Coulomb branch associated to a dismemberment of $Q$. Thus, for any quiver, it is natural to consider the canonical dismemberment where the set of edges joining each pair of vertices gives its own component.  It would be interesting to give a more geometric description of the Coulomb branches corresponding to these components; such a description is known in the case of a quiver with exactly one edge, using the theory of bow varieties (see \cite[\S 3.1]{nakajimaCherkisBow2017}).  This is generalized in \cref{lem:bow}.

Some examples arising from quivers also suggest a connection between our results and relative Langlands duality, giving a candidate construction of the relative Langlands dual (in the sense of \cite{BenZviSakleredisVenkateshRelativeLanglandsDuality}) of the Hamiltonian $G$-variety $T^*\bfN$ for a finite dimensional $G$-representation $\bfN$, which we discuss in \cref{section: recovering the S-dual}. We also give a related conjecture, \cref{Recoverability Conjecture}, which informally states that one can recover the ring object corresponding to $\M(G,\bfN)$ on $\Gr_G$ from the $\M_C(G, 0)$-action on $\M(G,\bfN)$ and that our candidate construction of the relative Langlands dual agrees with the definition implicitly used in \cite[Section 8]{BenZviSakleredisVenkateshRelativeLanglandsDuality} (see in particular \cite[Conjecture 8.1.8]{BenZviSakleredisVenkateshRelativeLanglandsDuality} and \cite[Remark 8.3.2]{BenZviSakleredisVenkateshRelativeLanglandsDuality}) and stated explicitly by Nakajima in \cite{NakajimaSDualOfHamiltonianGSpacesandRelativeLanglandsDuality}. As explained in \cref{rem:Gaiotto-Witten}, our conjecture can be seen as an application of \cref{Intro Statement of Functoriality Theorem} to  the prescription of Gaiotto and Witten for this dual given in \cite[\S 4.3]{gaiottoSDuality2009}.  
%
\begin{Remark}
As we have already mentioned in \cref{We Think You Dont Need Gluability}, we believe that \cref{Intro Statement of Functoriality Theorem} holds without the gluability hypothesis. Assuming this is true, one can immediately obtain an isomorphism of the form \cref{Gluing Isomorphisms for Quiver} for an arbitrary dismemberment of an arbitrary quiver. 

We observe that if one has an isomorphism of the form \cref{Gluing Isomorphisms for Quiver} for $\check{Q}$ the finest possible dismemberment of an arbitrary quiver $Q$, then the Coulomb branch of any quiver gauge theory can be obtained from products of varieties of the form $\M(Q, (n_i,n_j))$ by applying the above Hamiltonian reduction procedure.
\end{Remark}
We also apply \cref{Intro Statement of Functoriality Theorem} to study the Coulomb branch of the \textit{fission} of a quiver; see \cref{Explosions of quiver gauge theories subsection} for more details.
\subsection{Corollaries for affine closure of $T^*(G/U_P)$}\label{Corollary for Partial Implosion}
We fix a parabolic subgroup $P$ of $G$, and let $U_P$ denote the unipotent radical of $P$, and $L$ a fixed Levi subgroup of $P$. Our main cases of interest will be the parabolic subgroups $P$ of the groups $\GL_{n}$ or $\SL_{n}$. Recall that, for such $G$, any parabolic subgroup of $G$ is conjugate to exactly one \textit{standard parabolic subgroup}: a group of block upper triangular matrices in $G$ with block sizes indexed by an ordered partition $n = m_1 + ... + m_l$ of $n$, which we now fix once and for all and denote $\vec{m}$.
In this case, it is not difficult to directly check that $P$ is a semi-direct product of its unipotent radical $U_P$ of block strictly upper triangular matrices with the \textit{standard Levi subgroup} $L$ of block diagonal matrices in $G$ with block sizes indexed by $\vec{m}$. 

The affine scheme $\overline{T^*(G/U_P)} := \Spec(\mathcal{O}(T^*(G/U_P))$ given by the spectrum of the functions on the cotangent bundle $T^*(G/U_P)$ is sometimes referred to as the \textit{affine closure} of $T^*(G/U_P)$ since $T^*(G/U_P)$ is quasi-affine and therefore is an open subset of $\overline{T^*(G/U_P)}$, and is also known as a \textit{partial implosion} for $G$.  See \cite{DancerGrimmingerMartensZhongComplexSymplecticContractionsand3dMirrors} for a more detailed discussion. Our second main result proves that $\overline{T^*(G/U_P)}$ is a Coulomb branch in the sense of \cite{BravermanFinkelbergNakajimaTowardsaMathematicalDefinitionOfCoulombBranchesOf3dNEquals4GaugeTheoriesII} when $G = \GL_{n}$ or $G = \SL_{n}$, which proves a conjecture of Bourget-Dancer-Grimminger-Hanany-Zhong \cite{BourgetDancerGrimmingerHananyZhongPartialImplosionsandQuivers}, see also \cite{DancerHananyKirwanSymplecticDualityandImplosions}. We state the $\GL_{n}$ variant of this conjecture here for the ease of exposition and state the modification required for the case $G = \SL_{n}$ in \cref{Full Statement Coulomb Branches of Particular Quivers}. 

For our fixed ordered partition $\vec{m}$, we let $Q_{\vec{m}}$ denote any quiver whose underlying graph is
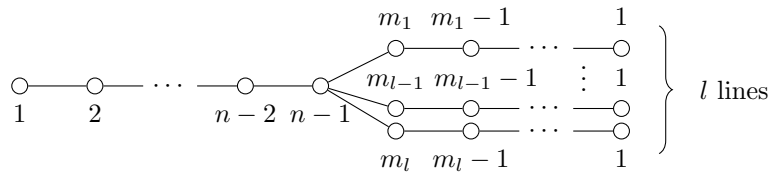
\begin{equation*}
\begin{tikzpicture}
    \tikzstyle{gaugered} = [circle, draw=black, inner sep=1pt, minimum size=6pt, label distance=1.5pt]

    \tikzstyle{gaugeblue} = [circle, draw=black, inner sep=1pt, minimum size=6pt, label distance=1.5pt]
    \tikzstyle{gaugepurple} = [circle, draw=black, inner sep=1pt, minimum size=6pt, label distance=1.5pt]
    \tikzstyle{gauge} = [circle, draw, inner sep=1pt, minimum size=6pt, label distance=1.5pt]
    
    \node[] (0) at (0.2,0) {\textcolor{white}{Q m := }};
    \node[gaugered, label=below:{$1$}] (1) at (1,0) {};
    \node[gaugered, label=below:{$2$}] (2) at (2,0) {};
    \node (3) at (3,0) {$\cdots$};
    \node[gaugered, label=below:{$n-2$}] (4) at (4,0) {};
    \node[gaugered, label=below:{$n-1$}] (5) at (5,0) {};
    
    \draw (1)--(2)--(3)--(4)--(5);
    
    \node[gaugeblue, label=above:{$m_1$}] (6u) at (6,0.5) {};
    \node[gaugeblue, label=above:{$m_1 - 1$}] (6u2) at (7,0.5) {};
    \node (7u) at (8,0.5) {$\cdots$};
    \node[gaugeblue, label=above:{$1$}] (8u) at (9,0.5) {};
    \node[gaugeblue, label=above:{$m_{l - 1}$}] (6m) at (6,-0.3) {};
    \node[gaugeblue, label={[xshift=0.2cm]above:{$m_{l - 1} - 1$}}] (6m2) at (7,-0.3) {};
    \node (7m) at (8,-0.3) {$\cdots$};
    \node (7x) at (8.5, 0.2) {$\vdots$};
    \node[gaugeblue, label=above:{$1$}] (8m) at (9,-0.3) {};
    \node[gaugeblue, label=below:{$m_l$}] (6d) at (6,-0.6) {};
    \node[gaugeblue, label=below:{$m_l - 1$}] (6d2) at (7,-0.6) {};
    \node (7d) at (8,-0.6) {$\cdots$};
    \node[gaugeblue, label=below:{$1$}] (8d) at (9,-0.6) {};
    
    \draw (5)--(6u)--(6u2)--(7u)--(8u)
          (5)--(6m)--(6m2)--(7m)--(8m)
          (5)--(6d)--(6d2)--(7d)--(8d);
    
    \draw [decorate, decoration={brace, amplitude=5pt}, xshift=0, yshift=0]
    (9.5,0.8)--(9.5,-0.9) node [black, midway, xshift=1cm, yshift=0] {$l \text{ lines}$};
\end{tikzpicture}
\end{equation*}
where to each vertex, we label by a nonnegative integer. We recall the representation $\bfN_{Q_{\vec{m}}}$ of the group $\GL_{Q_{\vec{m}}} := \prod_{v} \GL_{n_v}$ defined in \cref{Dismemberments of Quiver Gauge Theories Subsection}, where the product is taken over the vertices of $Q_{\vec{m}}$. Our final main result states that $\overline{T^*(\GL_{n}/U_P)}$ is a Coulomb branch for the quiver gauge theory with quiver $Q_{\vec{m}}$:

\begin{Theorem}\label{final main theorem}
There is an isomorphism of affine varieties $\overline{T^*(\GL_{n}/U_P)} \cong \mathcal{M}_C(\GL_{Q_{\vec{m}}}, \bfN_{Q_{\vec{m}}})$. 
\end{Theorem}


From \cref{final main theorem} and the $\SL_{n}$ variant of \cref{Full Statement Coulomb Branches of Particular Quivers}, one can immediately derive the following corollary from known facts on the geometry of Coulomb branches: 
\begin{Corollary}\label{Corollaries on Geometry of Coulomb Branches}If $G = \SL_{n}$ or $G = \GL_{n}$ then \begin{enumerate}
    \item The ring of functions on $T^*(G/U_P)$ is finitely generated (by \cite[Proposition 6.8]{BravermanFinkelbergNakajimaTowardsaMathematicalDefinitionOfCoulombBranchesOf3dNEquals4GaugeTheoriesII}).
    \item The variety $\overline{T^*(G/U_P)}$ has \textit{symplectic singularities} in the sense of \cite{BeauvilleSymplecticSingularities}; in particular, it has Gorenstein rational singularities (by \cite[Theorem 1.1]{BellamyCoulombBranchesHaveSymplecticSingularities} or \cite[Theorem 1]{WeekesQuiverGaugeTheoriesandSymplecticSingularities}).
    \item There is a stratification of $\overline{T^*(G/U_P)}$ by finitely many holomorphic symplectic subvarieties (by \cite[Theorem 2.3]{KaledinSymplecticSingularitiesfromthePoissonPointofView}).
    \item The variety $\overline{T^*(G/U_P)}$ is independent of the ordering of the ordered partition $\vec{m}$ of $n$.
\end{enumerate}
\end{Corollary}

Let us mention in passing that the finite generation of the ring of functions on $T^*(G/U_P)$ is nontrivial, as this ring is given by the invariants of a finitely generated algebra by the unipotent group $U_P$ and thus need not a priori be finitely generated. In the case where $P$ is a Borel subgroup, this finite generation was proved for $G = \SL_{n}$ in \cite{DancerKirawanSwannImplosionForHyperKahlerManifolds}, and was proved for general reductive groups in \cite{GinzburgRicheDifferentialOperatorsOnBasicAffineSpaceandtheAffineGrassmannian}. The proof of \cite{GinzburgRicheDifferentialOperatorsOnBasicAffineSpaceandtheAffineGrassmannian} shows that, in essentially all cases, the functions on $T^*(G/U_P)$ is not generated by the \lq na\"ive guess\rq{} of generating set: the pullback of functions on the base $G/U_P$ along with the global vector fields $\LG \oplus \mathfrak{l}$ obtained from the $G \times L$ action on $G/U_P$.  On the other hand, results of Weekes \cite{WeekesCoulomb} construct an explicit generating set for Coulomb branches of quiver gauge theories generalizing that of the Yangian; thus \cref{final main theorem} gives such a presentation on the ring of functions on $T^*(\GL_n/U_P)$. 

We also observe that \cref{Corollaries on Geometry of Coulomb Branches}(2) implies a generalization of the \textit{Ginzburg-Kazhdan conjecture} \cite{GinzburgKazhdanDifferentialOperatorsOnBasicAffineSpaceandtheGelfandGraevAction}, originally proved for $G = \SL_{n}$ in \cite{JiaTheGeometryOfTheAffineClosureOfCotangentBundleOfBasicAffineSpaceForSLn} and in \cite{GannonProofOftheGinzburgKazhdanConjecture} for general reductive groups, to the parabolic setting when $G$ is $\SL_{n}$ or $\GL_{n}$. 

\subsection{Acknowledgments} We thank David Ben-Zvi, Alexander Braverman, Philip Boalch, Peter Crooks, Sanath Devalapurkar, Davide Gaiotto, Benjamin Gammage, Amihay Hanany, Justin Hilburn, Hiraku Nakajima, Morgan Opie, Brian Shin, Aaron Slipper, and Harold Williams for interesting and useful discussions. The first author would also like to thank Perimeter Institute, where much of this project was completed, for their hospitality. B.W. is supported by Discovery Grant RGPIN-2024-03760 from the
  Natural Sciences and Engineering Research Council of Canada.
This research was also supported by Perimeter Institute for Theoretical Physics. Research at Perimeter Institute is supported by the Government of Canada through the Department of Innovation, Science and Economic Development and by the Province of Ontario through the Ministry of Research and Innovation.

\section{Proof of Coulomb branch functoriality}
\label{sec:proof}

Throughout the paper, we will use $G$ to denote a complex reductive group and $\mathbf{N}$ a representation of $G$.  Fix a reductive group $\tG$ equipped with inclusion of $G$ as a normal subgroup such that $F=\tG/G$ is a torus, and further assume that the representation $\bfN$ extends to $\tG$ in the sense that the following diagram
\[\begin{tikzcd}
	\tG && {GL(\bfN)} \\
	& G
	\arrow["\imath", from=1-1, to=1-3]
	\arrow[hook, from=2-2, to=1-1]
	\arrow[hook, from=2-2, to=1-3]
\end{tikzcd}\] commutes. One important example that makes sense for any group is when $\tilde{G}=G\times \Gm$ with $\Gm$ acting on $\mathbf{N}$ with the usual scalar multiplication. In what follows, we will use the notation \newcommand{\Gmsc}{\mathbb{G}_m^{\mathrm{sc}}}$\Gmsc := \mathbb{G}_m$ to denote the copy of $\mathbb{G}_m\subset GL(\bfN)$ given by scalar matrices.  We let $W$ denote the common Weyl group of $G$ and $\tG$.  

For any affine scheme $X$, we let $\mathcal{O}(X)$ denote the ring of global functions on $X$. For any group scheme $\mathcal{G}$ acting on  $X$, we use the notation \[X\sslash \mathcal{G} := \Spec(\mathcal{O}(X)^{\mathcal{G}})\] to denote the spectrum of the $\mathcal{G}$-invariant functions on $X$, which is an affine scheme. As a particular case of this construction, we obtain the affine variety $\mathfrak{c}_G := \LG\sslash G$.

Recall that by \cite[Lemma 5.3]{BravermanFinkelbergNakajimaTowardsaMathematicalDefinitionOfCoulombBranchesOf3dNEquals4GaugeTheoriesII}, the Coulomb branch $\mathcal{M}_C(\tG, \mathbf{N})$ admits a flat map \begin{equation}\label{Flat Map for MCtildeG}\mathcal{M}_C(\tG, \mathbf{N}) \to \tcG := \LtG\sslash\tG \end{equation}and, moreover, by \cite[Section 3(v)]{BravermanFinkelbergNakajimaTowardsaMathematicalDefinitionOfCoulombBranchesOf3dNEquals4GaugeTheoriesII}, acquires a Hamiltonian action of $F^{\vee}$ which, by \cite[Proposition 3.18]{BravermanFinkelbergNakajimaTowardsaMathematicalDefinitionOfCoulombBranchesOf3dNEquals4GaugeTheoriesII}, has the property that its Hamiltonian reduction is $\M_C(G, \bfN)$.

\begin{Definition}
	We let $\mathcal{M}_C(\tG, G, \mathbf{N})$ be the scheme quotient $\mathcal{M}_C(\tG, \mathbf{N})\sslash F^{\vee}$, or equivalently the affine scheme whose functions are the $F^{\vee}$-invariant functions on $\mathcal{M}_C(\tG, \mathbf{N})$.  
\end{Definition} 

Using the map induced by the $F^{\vee}$-equivariance of \cref{Flat Map for MCtildeG}, we obtain a map $\mathcal{M}_C(\tG, G, \mathbf{N})\to \tcG$. Moreover, composing \cref{Flat Map for MCtildeG} with the natural quotient map, we see that the affine scheme $\mathcal{M}_C(\tG, G, \mathbf{N})$ is a flat family over the Lie algebra $\mathfrak{f}$ whose fiber over $0$ is the usual Coulomb branch $\mathcal{M}_C(G, \mathbf{N})$. 

For any scheme $Y \to Z_0$ over $Z_0$ and map $Z \to Z_0$, we write $Y_Z := Y \times_{Z_0} Z$ for the corresponding base change, viewed as a scheme over $Z$. We'll use several times the natural functoriality in the choice of $\tG$, which follows immediately for functoriality in equivariant homology:
\begin{Lemma}\label{Functoriality in tG'}
If $\tG'\subset \tG$ is a connected subgroup containing $G$, then $\mathcal{M}_C(\tG', G, \mathbf{N}) \cong \mathcal{M}_C(\tG, G, \mathbf{N})_{\LtG'\sslash\tG}$.  In particular, $\mathcal{M}_C(G, \mathbf{N})\cong \mathcal{M}_C(\tG, G,\mathbf{N})_{\cG}.$
\end{Lemma}
We hereafter use the notation 
 \begin{align*}
\M(G) &:= \mathcal{M}_C(G, 0)&\M(G, \mathbf{N}) &:= \mathcal{M}_C(G, \mathbf{N})  \\  
\tM(G) & := \mathcal{M}_C(\tG, G, 0)\cong \mathcal{M}_C(G,0)\times \LF&\tM(G, \mathbf{N}) & := \mathcal{M}_C(\tG , G, \mathbf{N}) 
 \end{align*}
 to avoid clutter.


\subsection{Gluing in massive case}\label{Coulomb Branch as Subring for Massive Representations Subsection} 
In this section, we will set some notation and use it to recall a result of Teleman \cite{TelemanTheRoleofCoulombBranchesin2DTheory} (see \cref{Functions on Massive Coulomb Branch as Coproduct of Open Embeddings}) on the ring of functions of certain Coulomb branches. We let $\tilde{T}$ define a maximal torus of $\tG$, let $T := G \cap \tilde{T}$ denote the corresponding maximal torus in $G$, and let $\LT$ and $\tilde{\LT}$ denote the respective Lie algebras. Let $\varphi_1,\dots, \varphi_d$ be the weights of the represention $\mathbf{N}$, considered as a multiset of linear functions on $\LtT$.  Each such weight can also be considered as a cocharacter $\Gm\to T^{\vee}$ to the dual torus, which we denote $a\mapsto a^{\varphi_i}$.  Thus, any non-zero weight defines a rational map $\LtT\dashrightarrow T^{\vee}$ defined by $\zeta\mapsto  {\varphi_i}(\zeta)^{\varphi_i}$.  The product of these functions defines a  rational map \[\epsilon_{\mathbf{N}}: \LtT \dashrightarrow T^{\vee} \qquad \epsilon_{\mathbf{N}}(\zeta)=\prod_{i=1}^d {\varphi_i}(\zeta)^{\varphi_i}. \]  
Since $\mathbf{N}$ is a $\tG$-representation, the multi-set $\{\varphi_1,\dots, \varphi_d\}$ is $W$-invariant, and so $\epsilon_{\mathbf{N}}$ is equivariant. The functoriality of Borel-Moore homology gives a birational map \[\tM(G)_{\LtT} \xrightarrow{} \tM(T)\] (say, by \cite[Lemma 5.17]{BravermanFinkelbergNakajimaTowardsaMathematicalDefinitionOfCoulombBranchesOf3dNEquals4GaugeTheoriesII}) and so $\epsilon_{\mathbf{N}}$ defines a $W$-equivariant rational map \begin{equation}\label{Rational Map Determined by Epsilon}\tM(G)_{\LtT} \xrightarrow{} \tM(T) = T^{\vee}\times \LtT\dashrightarrow T^{\vee}\times \LtT = \tM(T) \xleftarrow{} \tM(G)_{\LtT}\end{equation} where the dashed right arrow is given by the formula $(x, \zeta) \mapsto (x\epsilon_{\mathbf
N}(\zeta), \zeta)$. 

Let $\LtTc$ be the subset of $\LtT$ where all weights of $\mathbf{N}$ and roots of $\tG$ are non-zero, as functions on $\LtT$, and let $\LtGc$ denote the $G$-saturation of this set.  
  Of course, we have $\tcGc := \LtGc\sslash \tG\cong \LtTc/W.$  By \cite[Lemma 5.13]{BravermanFinkelbergNakajimaTowardsaMathematicalDefinitionOfCoulombBranchesOf3dNEquals4GaugeTheoriesII}, we have \begin{equation}
    \tM(G,\bfN)_{\LtGc\sslash G}\cong \tM(G)_{\LtGc\sslash G}\cong (\LtTc\times T^{\vee})/W.
\end{equation}  Since the rational section $\epsilon_{\mathbf{N}}$ is $W$-equivariant and has no zeros or poles on this open set, the formula \cref{Rational Map Determined by Epsilon} defines a regular map $\birat_{\mathbf{N}}^{\circ}\colon \tM(G)_{\tcGc}\to \tM(G)_{\tcGc}$, and thus a rational map $\birat_{\mathbf
N} \colon \tM(G)\to \tM(G)$. Since $\birat_{\mathbf{N}}^{\circ}$ is given by multiplication by a section and $\tM(G)$ is an abelian group scheme, the map $\birat_{\mathbf
N}$ commutes with the action of $\tM(G)$ on itself by left/right multiplication.

Let $\mathcal{F}(\tG, G, \mathbf{N})$ denote the ring of functions on $\tM(G,\bfN)$. 
\begin{Definition}\label{Coulomb Gluing Property Algebraically}
	We say that the triple $(\tG, G, \mathbf{N})$ has the {\bf Coulomb gluing property} if the ring of functions $\mathcal{F}(\tG, G, \mathbf{N})$ is the subset of regular functions on $\mathcal{F}(\tG, G, 0)$ such that the pullback by the rational map $\birat_{\mathbf N}$ is again regular. 
\end{Definition}

Translating \cref{Coulomb Gluing Property Algebraically} into the language of affine schemes, we immediately see that $(\tilde{G}, G, \bfN)$ has the Coulomb gluing property if and only if the deformed Coulomb branch $\tM(G,\bfN)$ is the pushout in the category of 
        affine schemes of 
        the diagram  
\begin{equation}\label{Pushout diagram}
    \begin{tikzcd}
	{\tM(G)_{\LtGc\sslash \tG}}  & {\tM(G)} \\
	{\tM(G)} & {\tM(G,\bfN)}
	\arrow["\iota", from=1-1, to=1-2]
	\arrow["{\iota\circ \birat_{\mathbf{N}}^{\circ}}"', from=1-1, to=2-1]
	\arrow[from=1-2, to=2-2]
	\arrow["{\jmath_{\mathbf{N}}}"', from=2-1, to=2-2]
\end{tikzcd}
\end{equation}
where $\iota$ is the inclusion of the open subset $\tM(G)_{\LtGc\sslash G}$. 

With this notation, we may now restate a theorem of Teleman:
\begin{Theorem}[\mbox{\cite[Theorem 1]{TelemanTheRoleofCoulombBranchesin2DTheory}}]\label{Functions on Massive Coulomb Branch as Coproduct of Open Embeddings} The triple $(G \times \Gmsc, G, \mathbf{N})$ has the Coulomb gluing property.
\end{Theorem}

\begin{Remark}\label{K-theory 1}
    A similar description of the K-theoretic (4-dimensional) Coulomb branch can also be given.  The primary difference is that the base is $\tilde{T}/W$ instead of $\LtT/W$, and the rational section is given by \[\lambda_{\mathbf{N}}: \LtT \dashrightarrow T^{\vee} \qquad \varepsilon_{\mathbf{N}}(t)=\prod_{i=1}^d (1-{\varphi_i}(t)^{-1})^{\varphi_i}, \]  where $\varphi_i\colon T\to \Gm\subset \mathbb{A}^1$ now denotes a multiplicative rather than additive character.  We can say that $(G \times \Gmsc, G, \mathbf{N})$ has the {\bf K-theoretic Coulomb branch gluing property} if the $K$-theoretic Coulomb branch is a similar pushout in the category of affine schemes.  
\end{Remark}

\subsection{Functoriality and gluing} In fact, a refinement of Teleman's proof of \cref{Functions on Massive Coulomb Branch as Coproduct of Open Embeddings} can be used to show a functorial upgrade of \cref{Functions on Massive Coulomb Branch as Coproduct of Open Embeddings}. 
Let $\tH\supset H$ be another pair of reductive groups with $\tH/H=F_H$ a torus. As in \cref{Coulomb Branch as Subring for Massive Representations Subsection}, we write \[\tcH := \LtH\sslash \tH\text{, and } \tcHc := \LtH^{\circ}\sslash \tH.\]
Let $\mapFromHToG: \tH \to \tG$ be a map of reductive groups inducing a map $H\to G$; we will occasionally refer to this as a map of pairs of algebraic groups \begin{equation}\label{Map of Pairs for Algebraic Groups}(\tH, H) \xrightarrow{\mapFromHToG} (\tG, G)\end{equation} induced by $\mapFromHToG$.

We choose a maximal torus $\tT_{H}\subset \tH$ and let $T_H=H\cap \tT_{H}$.  To avoid confusion, we denote the tori in $\tG,G$ chosen in \cref{Coulomb Branch as Subring for Massive Representations Subsection} by $\tT_G, T_G$.  Without loss of generality, we can assume that $\mapFromHToG(\tT_{H})\subset \tT_G$. We have an induced map of affine varieties \(\overline{\mapFromHToG}\colon \tcH  \xrightarrow{} \tcG\), and naturally obtain a map \begin{equation}\label{Base Change Coproduct}
        \tM(G)_{\tcH} \underset{\tM(G)_{\tcHc}}{\coprod} \tM(G)_{\tcH} \to \tM(G, \bfN)_{\tcH}
    \end{equation} induced by taking the base change of the diagram \cref{Pushout diagram} by $\overline{\mapFromHToG}$.

\begin{Definition}\label{def:gluable}
    We call a map of pairs \cref{Map of Pairs for Algebraic Groups} {\it gluable} for a representation $\tG \xrightarrow{} \GL(\bfN)$ if the representation has no pair of weights $\xi_1,\xi_2$ for $\tT_G$ such that \begin{enumerate}
  \item $\xi_1|_{\tT_H}=\alpha \xi_2|_{\tT_H}$ for some $\alpha\in \mathbb{Q}$ and
	\item for some cocharacter $\mu$ of $T_G$, the signs of the integers $\langle \xi_1|\mu\rangle $ and $ \langle \xi_2 | \mu\rangle $ are opposite, that is, $\langle \xi_1|\mu\rangle  \langle \xi_2 | \mu\rangle <0$.  
\end{enumerate}  

If the above conditions hold and $\tH = H$ and $\tG = G$, we will say the map $\mapFromHToG$ is \textit{gluable} for $\bfN$. Finally, we say a representation $\bfN$ of $G$ is \textit{gluable} for $\bfN$ if the identity map $G \to G$ is gluable for $\bfN$.
\end{Definition}

\begin{Remark}
    Observe that a representation $\bfN$ of $G$ is gluable if and only if there are no nonzero weights $\xi_1, \xi_2$ of $\bfN$ such that $\xi_1 = -\alpha\xi_2$ for some \textit{positive} $\alpha \in \mathbb{Q}$. This condition has appeared previously in the literature when $G$ is a torus, see \cite[Lemma 4.18]{ChanLeung3DMirrorSymmetryisMirrorSymmetry}.
\end{Remark}

With this notation, our extension of \cref{Functions on Massive Coulomb Branch as Coproduct of Open Embeddings} can be stated precisely as follows: 

\begin{Theorem}\label{Base Change of Massive Coulomb Branch is Coulomb Branch of Massive Base Changes}
The map \cref{Base Change Coproduct} is an isomorphism if the map of pairs \cref{Map of Pairs for Algebraic Groups} is gluable.  
Moreover, in this case, both of the triples $(\tG, G, \bfN)$ and $(\tH, H, \bfN)$ satisfy the Coulomb gluing property.
\end{Theorem}

Before giving the proof of this result, let us note a couple of consequences. First, observe that if the image of $\iota$ contains the scalar matrices $\Gm\cdot I$, then we can only have $\xi_1=\alpha\xi_2$ as weights of $\tG$ if $\alpha=1$.  In this case, $\langle \xi_1|\mu\rangle = \langle \xi_2 | \mu\rangle$ for any cocharacter $\mu \colon \Gm \to G$, so the inclusion is gluable for $\tH := \tG$.  In particular, \cref{Base Change of Massive Coulomb Branch is Coulomb Branch of Massive Base Changes} gives an extension of Teleman's \cref{Functions on Massive Coulomb Branch as Coproduct of Open Embeddings}.

Moreover, using \cref{Base Change of Massive Coulomb Branch is Coulomb Branch of Massive Base Changes}, one can easily show that if the Coulomb gluing property holds for some triple, it holds after shrinking the gauge group; indeed, the following result follows immediately from the fact that the conditions of \cref{def:gluable} automatically hold for weights if they hold for a smaller group:

\begin{Corollary}\label{Shrinking G}
Assume that $G' \subseteq \tG$ is a connected subgroup containing $G$.  If $(\tG,G',\bfN)$ has the Coulomb gluing property, then $(\tG,G,\bfN)$ does as well. 
\end{Corollary}

Finally, applying \cref{Functions on Massive Coulomb Branch as Coproduct of Open Embeddings} for the triple $(\tilde{G} \times \Gmsc, \tilde{G}, \bfN)$ and \cref{Shrinking G}, we obtain that the Coulomb gluing property for any triple can be obtained by extending $\tilde{G}$:

\begin{Corollary}\label{Adding on Gm To Any Triple Guarantees Gluability}
For any triple $(\tG,G,\bfN)$, the induced triple $(\tG\times \Gmsc,G,\bfN)$ has the Coulomb gluing property.
\end{Corollary}

Hereafter, we use the notation $\Gr_G$ for the affine Grassmannian for $G$, and use similar notation for any reductive group. When clear from context, we will also omit the subscript from the notation and write $\Gr := \Gr_G$.

\begin{proof}[Proof of \cref{Base Change of Massive Coulomb Branch is Coulomb Branch of Massive Base Changes}]
  Consider the map \[D: \tA(G) \oplus \tA(G) \xrightarrow{} \tA(G)^{\circ}=\mathcal{O}(\tcGc)\otimes_{\mathcal{O}(\tcG)}\tA(G)\] given by the difference of the inclusion map and the inclusion map twisted by the automorphism $\birat_{\bfN}^{\circ}$. The triple $(\tG,G,\bfN)$ has the Coulomb gluing property if and only if we have an exact sequence \begin{equation}\label{SES of Functions for Coulomb Branch}0 \to \tA(G, \bfN) \xrightarrow{i} \tA(G) \oplus \tA(G) \xrightarrow{D} \tA(G)^{\circ}\end{equation} of $\mathcal{O}(\tcG)$-modules. Moreover, the map \cref{Base Change Coproduct} is an isomorphism if and only if the sequence \begin{equation}\label{SES of Functions for Coulomb Branch2}0 \to \tA(G, \bfN)_{\tcH} \xrightarrow{i_{\tcH}} \tA(G)_{\tcH} \oplus \tA(G) _{\tcH}\xrightarrow{D_H} \tA(G)^{\circ}_{\tcH}\end{equation} is exact.  This sequence is filtered by the Schubert stratification of the affine Grassmannian, and is exact if  the induced sequence on associated graded is exact.  
  
  On $\Gr$, we have the trivial bundle $\mathcal{T}^0=\operatorname{Gr}\times \bfN(\mathcal{O})$ and the balanced product  \[\mathcal{T}^1=G(\mathcal{K})\times_{G(\mathcal{O})} \bfN(\mathcal{O})=\{(gG(\mathcal{O}),n(t))\in \operatorname{Gr}\times \bfN(\mathcal{K})\mid g^{-1}n(t)\in \bfN(\mathcal{O})\}.\]
    We can view $\mathcal{T}^0$ and $\mathcal{T}^1$ as subbundles of $\operatorname{Gr}\times \bfN(\mathcal{K})$, and define the BFN space $\mathcal{R}$ as their intersection.  
    Let $\lambda$ be a dominant coweight, $\operatorname{Gr}^{\lambda}$ be the $G(\mathcal{O})$-orbit containing $t^{\lambda}$ and $\mathcal{R}_{\lambda},\mathcal{T}^0_{\lambda}$ and $\mathcal{T}^1_{\lambda}$ the preimages in the relevant bundles.  By \cite[Lemma 2.2]{BravermanFinkelbergNakajimaTowardsaMathematicalDefinitionOfCoulombBranchesOf3dNEquals4GaugeTheoriesII},  the space $\mathcal{R}_{\lambda}$ is a bundle with finite codimension in $\mathcal{T}^i_{\lambda}$.  We will use below the  Euler classes
    \[\varphi_\ell=e(\mathcal{T}^1_{\lambda}/\mathcal{R}_{\lambda})\qquad \varphi_r=e(\mathcal{T}^0_{\lambda}/\mathcal{R}_{\lambda}) \]
    These are both products of linear factors in $\mathcal{O}(\LtT)$ corresponding to the weights $\mu$ where $\lambda$ has negative/positive value, respectively, with multiplicity $|\langle\lambda |\mu\rangle|$ (see, for example, \cite[\S 6.3]{TelemanTheRoleofCoulombBranchesin2DTheory}). Our conditions (1) and (2) above exactly guarantee that these have no common factors after restriction to $\mathcal{O}(\LtT_{H})$. 
    
    As discussed in \cite[\S 6(i)]{BravermanFinkelbergNakajimaTowardsaMathematicalDefinitionOfCoulombBranchesOf3dNEquals4GaugeTheoriesII}, component of the associated graded for $\lambda$ is a free module over $\mathcal{O}(\tc{G_{\lambda}})$, for the Levi subgroup $G_{\lambda}$ centralizing $\lambda$, generated by the class $[\mathcal{R}_{\lambda}]$. We have 
    \begin{equation}\label{Explicit Description of SES Maps}i([\mathcal{R}_{\lambda}])= (\varphi_\ell [\operatorname{Gr}^{\lambda}],\varphi_r[\operatorname{Gr}^{\lambda}]) \qquad D(f,g)=f-\frac{\varphi_\ell}{\varphi_r}g\end{equation} by the argument in \cite[\S 6.2]{TelemanTheRoleofCoulombBranchesin2DTheory}. From this, we see that $i$ is injective since $\varphi_r$ and $\varphi_\ell$ are non-zero. We moreover claim that $i_{\tcH}$ is injective if and only if there are no pairs of weights $\xi_1, \xi_2$ such that for some cocharacter $\mu$ of $T_G$, $\langle \xi_1|\mu\rangle  \langle \xi_2 | \mu\rangle <0$ and $\xi_1|_{\tH} = \xi_2|_{\tH} = 0$. Indeed, as the map on associated graded is a map from a free rank one $\mathcal{O}(\mathfrak{c}_{G_{\lambda}})$-module into a free rank two $\mathcal{O}(\mathfrak{c}_{G_{\lambda}})$-module, and the injectivity of this map is equivalent to the image of a generator being nonzero, and so the injectivity of $i_{\tcH}$ is equivalent to the condition that at least one of $\varphi_{\ell}$ and $\varphi_{r}$ is not in the kernel of $\mathcal{O}(\tcG) \to \mathcal{O}(\tcH)$.

    From \cref{Explicit Description of SES Maps}, we also see that the kernel of $D$ is the submodule of pairs $(f,g)$ such that $\varphi_rf=\varphi_\ell g$. Since $\mathcal{O}(\tc{G_{\lambda}})$ is a UFD, this is exactly the image of $i$ if and only if $\varphi_\ell$ and $\varphi_r$ are non-zero and have no common factors in $\mathcal{O}(\LtT_{H})$. Thus, if the sequence \cref{SES of Functions for Coulomb Branch2} is exact, we see that the conditions of (1) and (2) above are satisfied and, conversely, if these conditions are satisfied, then the sequence \cref{SES of Functions for Coulomb Branch2} is exact in the third term by our above analysis and exact in the second term since no pair of weights which satisfies (2) can restrict to zero as this would violate (1). This shows our first claim.

    We now show that these equivalent conditions imply the Coulomb branch gluing property for these triples. For the triple $(\tG, G, \bfN)$, we observe that condition (1) above is weaker than the condition that $\xi_1=\alpha \xi_2$ for $\alpha\in \mathbb{Q}$, so conditions (1-2) also hold if we set $\tH=\tG$; we therefore obtain that the map \cref{Base Change Coproduct} is an isomorphism for $\tH := \tG$, which is exactly the Coulomb gluing property. Finally, for the triple $(\tH, H, \bfN)$, we observe that condition (2) is weaker when $T_G$ is replaced with the maximal torus $T_H$: indeed, if $\xi_1 \circ \mapFromHToG = \alpha(\xi_2 \circ \mapFromHToG)$ for some $\alpha \in \mathbb{Q}$ and there is a cocharacter $\mu_H: \mathbb{G}_m \to T_H$ for which $\langle \xi_1|\mu_H\rangle  \langle \xi_2 | \mu_H\rangle <0$, then we immediately see that \[\langle \xi_1|\mapFromHToG\circ \mu_H\rangle  \langle \xi_2| \mapFromHToG \circ \mu_H\rangle = \langle \xi_1 \circ \mapFromHToG|\mu_H\rangle  \langle \xi_2\circ \mapFromHToG|\mu_H\rangle <0\] and so condition (2) does not hold. By our above analysis for the case $\tH = \tG$ and $H = G$, we see that the Coulomb gluing property holds for $(\tH, H, \bfN)$.
\end{proof}
\begin{Remark}\label{K-theory 2}
    The same result holds for $K$-theoretic Coulomb branches;  the only change is that the factors $\varphi_\ell,\varphi_r$ are now taken to be Euler classes in K-theory 
    \[\varphi_\ell=\prod_{\langle \lambda|\mu\rangle<0}(1-\mu^{-1})^{|\langle \lambda|\mu\rangle|}\qquad \varphi_r=\prod_{\langle \lambda|\mu\rangle>0}(1-\mu^{-1})^{|\langle \lambda|\mu\rangle|}\]
    Condition (2) again guarantees that these are coprime as elements of $\mathcal{O}(T)$, so the proofs of \cref{Base Change of Massive Coulomb Branch is Coulomb Branch of Massive Base Changes,Shrinking G,Adding on Gm To Any Triple Guarantees Gluability} carry through unchanged in the K-theoretic case as well.
\end{Remark}

\subsection{Action for general Coulomb branches}\label{Action for General Coulomb Branches}
Recall 
a result of Teleman \cite[Theorem 1(iv)]{TelemanCoulombBranchesforQuaternionicRepresentations} which says that the group scheme $\M(G)$ acts on $\mathcal{M}_C(G, \bfN)$ as a scheme over $\cG$. As with \cref{Functions on Massive Coulomb Branch as Coproduct of Open Embeddings}, we will extend this to the deformed case.

\begin{Corollary}\label{Action on General Coulomb Branch by Gluing}
    The action of $\tM(G)$ on itself lifts to an action of $\tM(G)$ on $\tM(G, \mathbf{N})$. More precisely, there is an induced map making the following diagram 
\[\begin{tikzcd}
	{\tM(G) \times_{\tcG}\tM(G)} & {\tM(G) } \\
	{\tM(G)  \times_{\tcG}\tM(G, \mathbf{N})} & {\tM(G, \mathbf{N})  }
	\arrow["{\mathrm{mult}}", from=1-1, to=1-2]
	\arrow["{\mathrm{id} \times \jmath_{\mathbf{N}}}", from=1-1, to=2-1]
	\arrow["{\jmath_{\mathbf{N}}}"', from=1-2, to=2-2]
	\arrow["\exists", dashed, from=2-1, to=2-2]
\end{tikzcd}\] commute, and this map gives rise to an action of $\M(G)$ on $\mathcal{M}(G, \mathbf{N})$. 
\end{Corollary}

\begin{proof}[Proof of \cref{Action on General Coulomb Branch by Gluing}]
If the result holds for $\tG$, then by \cref{Functoriality in tG'}, it holds for any smaller $\tG'\subset \tG$ that satisfies the same conditions. Therefore, by \cref{Adding on Gm To Any Triple Guarantees Gluability}, extending $\tG$ if necessary we assume that the triple $(\tG,G,\bfN)$ has the Coulomb gluing property.  

We have a canonical isomorphism \[(\tM(G) \times_{\tcG} \tM(G)) \underset{\tM(G)_{\tcGc} \times_{_{\tcGc}} \tM(G)_{\tcGc}}{\coprod} (\tM(G) \times_{\tcG} \tM(G)) \xrightarrow{\sim} \tM(G) \times_{\tcG} \tM(G, \mathbf{N})\] from \cref{Functions on Massive Coulomb Branch as Coproduct of Open Embeddings} and the flatness of $\tM(G)$ over $\tcG$. Under this isomorphism, the action is given by the coproduct of the multiplication map for the group $\tM(G)$ with itself--since the map $\birat^{\circ}_{\mathbf{N}}$ is equivariant with respect to this action, this map is well defined. 
\end{proof}


\subsection{Construction of the morphism}\label{Construction of Morphism for General Coulomb Branches Subsection}
In all of \cref{Construction of Morphism for General Coulomb Branches Subsection}, we assume that the map of pairs \cref{Map of Pairs for Algebraic Groups} is gluable for $\bfN$. 
Recall that $\mapFromHToG$ induces a map on affine Grassmannians. Using the fact that the functions on the Coulomb branch with $\bfN = 0$ is given by the Borel-Moore homology of the affine Grassmannian, one can construct a morphism $\M_C(\tG)_{\tcH} \to \M_C(\tH)$ of group schemes over $\tcH$; taking the map induced by the quotients by the respective flavor tori, we obtain a morphism \[\mapFromBaseChangedCoulombBranchforGtoCoulombBranchforH: \tM(G)_{\tcH} \to \tM(H)\] of group schemes over $\tcH$. 
Note that this does not depend on $H\hookrightarrow G$ being injective---if we consider a surjective map $G\to Q$, we obtain an injective map $\tM(Q)_{\tcG}\to \tM(G).$

In this section, we will generalize this construction to define a map \[\tilde{f}'_{\mapFromHToG, \mathbf{N}}: \tM(G, \mathbf{N})_{\tcH} \to \tM(H, \mathbf{N})\] which is $\tM(G)$-equivariant. 
To construct such a map, it suffices to construct a map from the domain of \cref{Base Change Coproduct}. Informally, our desired map can be described as gluing two copies of $\mapFromBaseChangedCoulombBranchforGtoCoulombBranchforH$. 
Since $\mapFromBaseChangedCoulombBranchforGtoCoulombBranchforH$ intertwines the respective automorphisms $\birat^{\circ}_{\mathbf{N}}$ for the groups $\tG$ and $\tH$, there is a well defined morphism \begin{equation}\label{Actual Content of Morphism}\mapFromBaseChangedCoulombBranchforGtoCoulombBranchforH \underset{\mapFromBaseChangedCoulombBranchforGtoCoulombBranchforH}{\coprod} \mapFromBaseChangedCoulombBranchforGtoCoulombBranchforH: \tM(G)_{\tcH} \underset{\tM(G)_{\tcHc}}{\coprod} \tM(G)_{\tcH} \to \tM(H) \underset{\tM(H)_{\tcHc}}{\coprod}\tM(H) \end{equation} and, since the canonical map \[\tM(H) \underset{\tM(H)_{\tcHc}}{\coprod}\tM(H) \xrightarrow{} \tM(H, \mathbf{N})\] is an isomorphism by \cref{Base Change of Massive Coulomb Branch is Coulomb Branch of Massive Base Changes}, we obtain a morphism \[\tilde{f}'_{\mapFromHToG, \mathbf{N}}: \tM(G, \mathbf{N})_{\tcH} \to \tM(H, \mathbf{N})\] which, by construction, is $\tM(G)_{\tcH}$-equivariant.

Using the $\tM(H)$-action on the codomain, we let \begin{equation}\label{Definition of Tildef}\tilde{f}_{\mapFromHToG, \mathbf{N}}: \tM(H) \times_{\tcG}^{\tM(G)} \tM(G, \mathbf{N}) \cong \tM(H) \times_{\tcH}^{\tM(G)_{\tcH}} \tM(G, \mathbf{N})_{\tcH} \to \tM(H, \mathbf{N})\end{equation} denote the induced map. We claim this map is an isomorphism:

 \begin{Theorem} \label{Restriction of Groups is Hamiltonian Reduction}
If the pair \cref{Map of Pairs for Algebraic Groups} is gluable, the map $\tilde{f}_{\mapFromHToG, \mathbf{N}}$ is an isomorphism. 
\end{Theorem}
Note that the case of the undeformed Coulomb branch \[f_{\mapFromHToG, \mathbf{N}}: \mathcal{M}(H) \times_{\cG}^{\M(G)}\mathcal{M}_C(G, \mathbf{N}) \cong  \mathcal{M}(H) \times_{\cH}^{\M(G)_{\cH}}\mathcal{M}_C(G, \mathbf{N})_{\cH}\to \mathcal{M}(H, \mathbf{N})\] 
is covered by this result; it is just the case where $\tG=G$ and $\tH=H$. In particular, \cref{Intro Statement of Functoriality Theorem} is a special case of \cref{Restriction of Groups is Hamiltonian Reduction}.

\begin{proof}[Proof of \cref{Restriction of Groups is Hamiltonian Reduction}]

To prove $\tilde{f}_{\mapFromHToG, \mathbf{N}}$ is an isomorphism, it suffices to show that $\mapFromBaseChangedCoulombBranchforGtoCoulombBranchforH \underset{\mapFromBaseChangedCoulombBranchforGtoCoulombBranchforH}{\coprod} \mapFromBaseChangedCoulombBranchforGtoCoulombBranchforH$ in \cref{Actual Content of Morphism} is an isomorphism. Since $\tM(H) \to \tcH$ is flat, we have an isomorphism \[(\tM(H) \times_{\tcH}\tM(G)_{\tcH})\underset{\tM(H)_{\tcHc}\times_{\tcHc}\tM(G)_{\tcHc}}{\coprod} (\tM(H) \times_{\tcH}\tM(G)_{\tcH}) \xrightarrow{\sim}  \tM(H) \times_{\tcH}(\tM(G)_{\tcH} \underset{\tM(G)_{\tcHc}}{\coprod} \tM(G)_{\tcH})\] and, after pulling back by this isomorphism, we see that the map induced by $\mapFromBaseChangedCoulombBranchforGtoCoulombBranchforH \underset{\mapFromBaseChangedCoulombBranchforGtoCoulombBranchforH}{\coprod} \mapFromBaseChangedCoulombBranchforGtoCoulombBranchforH$ and the $\tM(H)$-action is given by $\mapFromBaseChangedCoulombBranchforGtoCoulombBranchforH' \underset{\mapFromBaseChangedCoulombBranchforGtoCoulombBranchforH'}{\coprod} \mapFromBaseChangedCoulombBranchforGtoCoulombBranchforH'$ where $\mapFromBaseChangedCoulombBranchforGtoCoulombBranchforH'(h_1, g_1) := h_1\mapFromBaseChangedCoulombBranchforGtoCoulombBranchforH(g_1)$. Quotienting by $\tM(G)$ and using the fact that colimits commute with colimits, the fact that $\mapFromBaseChangedCoulombBranchforGtoCoulombBranchforH' \underset{\mapFromBaseChangedCoulombBranchforGtoCoulombBranchforH'}{\coprod} \mapFromBaseChangedCoulombBranchforGtoCoulombBranchforH'$ is an isomorphism immediately follows from the fact that the map \[\tM(H) \times_{\tcH}^{\tM(G)_{\tcH}}\tM(G)_{\tcH} \xrightarrow{} \tM(H)\] induced by $\mapFromBaseChangedCoulombBranchforGtoCoulombBranchforH'$ is an isomorphism. This proves that $\tilde{f}_{\mapFromHToG, \bfN}$ is an isomorphism, as desired.
\end{proof}
\section{Consequences for Coulomb Branches of Quiver Gauge Theories}\label{Consequences for Coulomb Branches of Quiver Gauge Theories}
\subsection{Proof of \cref{gluing for quiver case}}\label{Proof of Gluing for Quiver Case Subsection}We now prove \cref{gluing for quiver case}. Since changing the orientation of any edge of the quiver leaves the Coulomb branch unchanged, \cref{gluing for quiver case} will follow if, for some choice of orientation, the inclusion $\GL(V_Q) \to \GL(V_{\check{Q}})$ is gluable. We now show this:
\begin{Lemma}\label{lem:gluable quiver}
The map $\GL(V_Q) \to \GL(V_{\check{Q}})$ is gluable if $Q$ has no loops and all pairs of parallel edges in $Q$ have the same orientation and lift to parallel edges in $\check{Q}$.
\end{Lemma}

\begin{proof}
Observe that since there are no edge loops, every weight of $V_{\check{Q}}$ is completely determined by two distinct vertices and an edge $v \to v'$ between them, in which case the weight has the form \[D_{\check{Q}} \to \mathbb{G}_m^{n_v} \times \mathbb{G}^{n_{v'}}_m \xrightarrow{(D, D') \mapsto D_{ii}^{-1}D_{jj}'} \mathbb{G}_m\] for some $i \in \{1, 2, ..., n_v\}$ and $j \in \{1, 2, ..., n_{v'}\}$.

If two weights satisfy (1) in \cref{def:gluable}, that is, are multiples of each other, they must correspond to parallel edges and the same choice of $i$ and $j$.  By the assumption that parallel edges always have the same orientation, these weights are equal.  By the assumption that parallel edges lift to parallel edges, the corresponding weights for $\GL(V_{\check{Q}})$ are also equal, and thus cannot satisfy (2).
\end{proof}
\subsection{Fission of quiver gauge theories}\label{Explosions of quiver gauge theories subsection} There is another operation on quivers whose effect on the associated Coulomb branches we can naturally describe using \cref{Intro Statement of Functoriality Theorem} and whose effect on the associated Higgs branches has been previously studied by Boalch \cite{BoalchSimplyLacedIsomonodromySystems}, \cite{BoalchIrregularConnectionsandKacMoodyRootSystems}: namely, the \textit{fission} (or \textit{explosion}) of vertices.  For a quiver and dimension vector $(Q,\mathbf{n})$, let $\ell_i$ be a second assignment of integers $\geq 1$ to the vertices $Q_0$. The fission $Q^{\circledast,\boldsymbol{\ell}}$ is the quiver obtained by replacing each vertex $i$ by $\ell_i$ vertices, and for each edge $i\to j$, adding an edge from every preimage of $i$ to a preimage of $j$.  In particular, if we begin with a single edge $i\to j$, and choose $(\ell_i,\ell_j)$, then $Q^{\circledast,\boldsymbol{\ell}}$ is the complete bipartite graph for $(\ell_i,\ell_j)$.  A fission of the dimension vector $\mathbf{n}$ is a dimension vector $\mathbf{n}^{\circledast}$ on $Q^{\circledast,\boldsymbol{\ell}}$ such that, for every $i \in Q_0$, $\mathbf{n}(i) = \sum_{i'}\mathbf{n}^{\circledast}(i')$ where this sum varies over the preimages of $i$. For example,

\begin{equation*}
\begin{tikzpicture}
    \tikzstyle{gaugered} = [circle, draw=black, inner sep=1pt, minimum size=6pt, label distance=1.5pt]

    \tikzstyle{gaugeblue} = [circle, draw=black, inner sep=1pt, minimum size=6pt, label distance=1.5pt]
    \tikzstyle{gaugepurple} = [circle, draw=black, inner sep=1pt, minimum size=6pt, label distance=1.5pt]
    \tikzstyle{gauge} = [circle, draw, inner sep=1pt, minimum size=6pt, label distance=1.5pt]
    
    \node[] (0) at (-1,0) {$A_n :=$ };
    \node[gaugered, label=below:{$1$}] (1) at (0,0) {};
    \node[gaugered, label=below:{$2$}] (2) at (1,0) {};
    \node (3) at (2,0) {$\cdots$};
    \node[gaugered, label=below:{$n - 1$}] (4) at (3,0) {};
    \node[gaugepurple, label=below:{$n$}] (5) at (4,0) {};

    \node[] (0) at (5.5,0) {$\implies$ };
    \draw[->] (1)--(2);
    \draw[->](2)--(3);
    \draw[->](3)--(4);
    \draw[->](4)--(5);   
    
    \node[] (07) at (7.2,0) {$A_n^{\circledast,\boldsymbol{\ell}} =$ };
    \node[gauge, label=below:{$1$}] (17) at (8,0) {};
    \node[gauge, label=below:{$2$}] (27) at (9,0) {};
    \node (37) at (10,0) {$\cdots$};
    \node[gauge, label=below:{$n - 1$}] (47) at (11,0) {};
    

        \draw[->] (17)--(27);
    \draw[->](27)--(37);
    \draw[->](37)--(47);

    \node[gauge, label=above:{$m_1$}] (67u) at (12,0.7) {};
    \node at (12,0.4) {$\vdots$};
    \node[gauge] (67m) at (12,-0.1) {};
    \node[gauge, label=below:{$m_l$}] (67d) at (12,-0.5) {};
    \draw[->] (47)--(67u);
    \draw[->] (47)--(67m);
    \draw[->] (47)--(67d);
\end{tikzpicture} 
\end{equation*} gives a fission of the quiver $A_n$ with a fissioned dimension vector; this fissioned quiver has been studied in \cite{DancerHananyKirwanSymplecticDualityandImplosions}, \cite{BoalchIrregularConnectionsandKacMoodyRootSystems}, and  \cite{GannonWilliamsDifferentialOperatorsOnBaseAffineSpaceofSLnandQuantizedCoulombBranches}.
As a consequence of our above analysis, we obtain that the Coulomb branch of any quiver $Q$ without loops determines the Coulomb branch for any fission of $Q$:

\begin{Corollary}\label{Pulled Back Quiver By Node Corollary}
If $Q$ has no loops, then there is an isomorphism\[\M(Q^{\circledast,\boldsymbol{\ell}}, \mathbf{n}^{\circledast})\cong\M(G(Q^{\circledast,\boldsymbol{\ell}}, \mathbf{n}^{\circledast}))\times_{\mathfrak{c}_{G(Q,\mathbf{n})}}^{\M(G(Q,\mathbf{n}))}\M(Q, \mathbf{n})\] of affine schemes which is $\M(G(Q^{\circledast,\boldsymbol{\ell}}, \mathbf{n}^{\circledast}))$-equivariant.
\end{Corollary}

\begin{proof}
First, as changing orientation of a quiver does not change its Coulomb branch, changing the orientation of edges on $Q$ we may assume that any two parallel edges have the same orientation. In particular, any two weights of $Q$ which are scalar multiples of each other are equal. Moreover, the group $G(Q^{\circledast,\boldsymbol{\ell}},\mathbf{n}^{\circledast})$ is a Levi subgroup of $G(Q,\mathbf{n})$; in particular, a maximal torus of the former determines a maximal torus of the latter. Therefore, the inclusion is gluable, and so, applying \cref{Intro Statement of Functoriality Theorem}, we obtain our claim. 
\end{proof}

One can combine \cref{Pulled Back Quiver By Node Corollary} for the quiver $A_n$ with \cite[Theorem 1.3]{GinzburgKazhdanDifferentialOperatorsOnBasicAffineSpaceandtheGelfandGraevAction}, \cite[Theorem 2.2]{GannonCotangentBundleofParabolicBaseAffineSpaceandKostantWhittakerDescent}, and \cite[Theorem 2.11]{BravermanFinkelbergNakajimaRingObjectsIntheEquivariantDerivedSatakeCategoryArisingFromCoulombBranches} to rederive the main result of \cite[Theorem 1.4]{GannonWilliamsDifferentialOperatorsOnBaseAffineSpaceofSLnandQuantizedCoulombBranches} without appealing to the results of \cite{GinzburgRicheDifferentialOperatorsOnBasicAffineSpaceandtheAffineGrassmannian} or \cite{MaceratoLEviEquivariantRestrictionofSphericalPerverseSheaves}. (We do not give the details for this claim here.)

\subsection{Proof of \cref{final main theorem}}\label{Proof of Main Theorem 2 Subsection}

In this section, we prove \cref{final main theorem}, as well as its variant for $\SL_{n}$. To state it precisely, we first observe that, for any vector space $V$, $\GL(V)$ contains a canonical central subgroup, which acts on $V$ by scalar multiplication. When $Q$ is a quiver, we will denote this subgroup by $\mathbb{G}_m \subseteq \GL(V_Q)$; by light abuse of notation, we are omitting the dependence of this subgroup on the quiver $Q$. Note that this group acts trivially on $\bfN$, so the representation factors through the quotient $\GL_{Q_{\vec{m}}}/\mathbb{G}_m$.  With this terminology, we may now state our main theorem: 

\begin{Theorem}\label{Full Statement Coulomb Branches of Particular Quivers}
There are isomorphisms \[\mathcal{M}_C(\GL_{Q_{\vec{m}}}/\mathbb{G}_m, \bfN_{Q_{\vec{m}}}) \cong \overline{T^*(\SL_{n}/U_P)}\text{, } \mathcal{M}_C(\GL_{Q_{\vec{m}}}, \bfN_{Q_{\vec{m}}}) \cong \overline{T^*(\GL_{n}/U_P)}\] of affine varieties.
\end{Theorem}
In particular, this establishes \cref{final main theorem}.



The entirety of \cref{Proof of Main Theorem 2 Subsection} will be devoted to the proof of \cref{Full Statement Coulomb Branches of Particular Quivers}. The proof of \cref{Full Statement Coulomb Branches of Particular Quivers} will also require \cref{lemma that relevant Coulomb branches are Kostant sections}, whose proof is deferred to \cref{Appendix to Avoid Derived Satake in Main Body}.

For any reductive group $G$, we view $T^*G$ as a scheme over $\mathfrak{g}^*$ using the moment map from the left $G$-action, and let $\Kost{G}$ denote the preimage of a Kostant section $\mathcal{S}_{\mathsf{reg}}\cong \fc{G^{\vee}}\subset \LGd$ under this moment map to $\LGd$. By standard arguments (see for example \cite[Lemma 3.2.3(iii)]{GinzburgKazhdanDifferentialOperatorsOnBasicAffineSpaceandtheGelfandGraevAction}) if we choose a nondegenerate $G$-invariant bilinear form on $\mathfrak g$ we may identify \[\Kost{G} \cong G \times \cG \cong \mu_U^{-1}(\psi)/U\] where $U$ is the unipotent radical of a Borel of $G$ and $\psi$ is a generic character of the Lie algebra of $U$. 
We recall the following result, which was known physically \cite[Section 7.1]{DimofteGarnerStarShapedQuiver} and can be immediately derived mathematically by taking cohomology of the isomorphism of algebra objects in \cite[Theorem 2.11]{BravermanFinkelbergNakajimaRingObjectsIntheEquivariantDerivedSatakeCategoryArisingFromCoulombBranches}: 
\begin{Theorem}\label{BFN Ring Objects Theorem to Coulomb Branch}
For any nonnegative integer $n$, there is an isomorphism \begin{equation}\label{Isomorphism for PGLn}\M_C(\GL(V_{A_{n}})/\mathbb{G}_m, \bfN_{A_{n}}) \cong \Kost{\SL_n}\end{equation} of varieties over $\mathfrak{c}_{\PGL_n} \cong \LSL_n^*\sslash \SL_n$. 
\end{Theorem}

We will now give a small upgrade of \cref{BFN Ring Objects Theorem to Coulomb Branch}, which we state precisely in \cref{lemma that relevant Coulomb branches are Kostant sections}, which informally says that \begin{enumerate}
    \item Both varieties appearing in the isomorphism \cref{Isomorphism for PGLn} naturally acquire actions of $\M_C(\PGL_n, 0)$ for which this isomorphism is equivariant and 
    \item There is a similar equivariant isomorphism as in \cref{Isomorphism for PGLn} where $\PGL_n$ is replaced with other various type $A$ groups such as $\GL_n$ and more generally products of copies of various $\GL_m$.
\end{enumerate}
We remark that both of these points may be known to experts, but we are unable to find a precise formulation or proof of these facts in the literature. 


To state \cref{lemma that relevant Coulomb branches are Kostant sections} precisely, we set more notation. For $m_i$ as above, let $A_{\vec{m}}$ denote any quiver whose underlying graph is 
\begin{equation*}
\begin{tikzpicture}
    \tikzstyle{gaugered} = [circle, draw=black, inner sep=1pt, minimum size=6pt, label distance=1.5pt]

    \tikzstyle{gaugeblue} = [circle, draw=black, inner sep=1pt, minimum size=6pt, label distance=1.5pt]
    \tikzstyle{gaugepurple} = [circle, draw=black, inner sep=1pt, minimum size=6pt, label distance=1.5pt]
    \tikzstyle{gauge} = [circle, draw, inner sep=1pt, minimum size=6pt, label distance=1.5pt]
    
    \node[gaugeblue, label=above:{$m_1$}] (6u) at (8,0.7) {};
    \node[gaugeblue, label=above:{$m_1 - 1$}] (6u2) at (9,0.7) {};
    \node (7u) at (10,0.7) {$\cdots$};
    \node[gaugeblue, label=above:{$1$}] (8u) at (11,0.7) {};
    \node[gaugeblue, label=above:{$m_{l - 1}$}] (6m) at (8,-0.3) {};
    \node[gaugeblue, label={[xshift=0.2cm]above:{$m_{l - 1} - 1$}}] (6m2) at (9,-0.3) {};
    \node (7m) at (10,-0.3) {$\cdots$};
    \node (73) at (10.4,.4) {$\vdots$};
    
    \node[gaugeblue, label=above:{$1$}] (8m) at (11,-0.3) {};
    \node[gaugeblue, label=below:{$m_l$}] (6d) at (8,-0.7) {};
    \node[gaugeblue, label=below:{$m_l - 1$}] (6d2) at (9,-0.7) {};
    \node (7d) at (10,-0.7) {$\cdots$};
    \node[gaugeblue, label=below:{$1$}] (8d) at (11,-0.7) {};
    
    \draw (6u)--(6u2)--(7u)--(8u)
          (6m)--(6m2)--(7m)--(8m)
         (6d)--(6d2)--(7d)--(8d);
\end{tikzpicture}.
\end{equation*} Observe that if our ordered partition $\vec{m}$ contains one element (i.e. $\vec{m} = \begin{pmatrix} n \end{pmatrix}$),  we have $A_n = A_{\vec{m}}$. As above, our ordered partition $\vec{m}$ determines a standard Levi subgroup $L$ of $\GL_n$, consisting of block diagonal matrices; we let $\overline{L} \subseteq \PGL_n$ denote the image of this group under the quotient map and let $S(L) := L \cap \SL_n$. 

For any reductive group $G$, we let $G^{\vee}$ denote the Langlands dual group, and recall that there is an isomorphism of group schemes over $\mathfrak{c}_G$ \begin{equation}\label{BFM Iso}
    \M_C(G, 0) \cong J_{G^{\vee}}\end{equation} 
 that identifies $\M_C(G, 0)$ with the \textit{group scheme of universal centralizers} $J_{G^{\vee}}$, or in other words the centralizer of a Kostant slice in $G^{\vee}$, given by \cite[Theorem 2.12]{BezrukavnikovFinkelbergMirkovicEquivariantHomologyandKTheoryofAffineGrassmanniansandTodaLattices}. Since $T^*G^{\vee}$ is a Hamiltonian $G^{\vee}$-space, the restriction to a Kostant section naturally acquires an action of $J_{G^{\vee}}$, 
 and so using the isomorphism \cref{BFM Iso}, we see that $\M_C(G, 0)$ acts on $\Kost{G^{\vee}}$ for any reductive group $G$ by left multiplication. With this notation, we state our small upgrade of \cref{BFN Ring Objects Theorem to Coulomb Branch}:
\begin{Lemma}\label{lemma that relevant Coulomb branches are Kostant sections}
For any $n$, there is an isomorphism \begin{equation}\label{Coulomb Branch Corollary of BFN Isos}\M_C(\GL(V_{A_{n}})/\mathbb{G}_m, \bfN_{A_{n}}) \cong \Kost{\SL_n}\text{, respectively } \M_C(\GL(V_{A_{n}}), \bfN_{A_{n}}) \cong \Kost{\GL_n}\end{equation} which is equivariant for the action of $\M_C(\PGL_{n}, 0)$, respectively $\M_C(\GL_{n}, 0)$. More generally, there are isomorphisms \begin{equation}\label{Coulomb Branch Corollary of Disjoint BFN Isos}\M_C(\GL(V_{A_{\vec{m}}})/\mathbb{G}_m, \bfN_{A_{\vec{m}}}) \cong \Kost{S(L)}\text{, respectively } \M_C(\GL(V_{A_{\vec{m}}}), \bfN_{A_{\vec{m}}}) \cong \Kost{L}\end{equation} equivariant for the actions of $\M_C(\overline{L}, 0)$, respectively $\M_C(L, 0)$.
\end{Lemma}

The proof of \cref{lemma that relevant Coulomb branches are Kostant sections} is given in \cref{Appendix to Avoid Derived Satake in Main Body}.

\begin{proof}[Proof of \cref{Full Statement Coulomb Branches of Particular Quivers}]
We prove the former claim; the latter is proved completely analogously. Since an isomorphism of groups induces an isomorphism of Coulomb branches (by our above functoriality results, say) we may assume that $L$ is a standard Levi subgroup. 
Applying \cref{Intro Statement of Functoriality Theorem} for the map \[\GL_{Q_{\vec{m}}}/\mathbb{G}_m \to (\GL_{A_{\vec{m}}}\times \GL_{A_{n}})/\mathbb{G}_m\] one obtains an isomorphism \begin{equation}\label{Coulomb Branch for Exploded and Glued Quivers}\M_C(\GL_{Q_{\vec{m}}}/\mathbb{G}_m, Q_{\vec{m}}) \cong (\M_C(\GL_{A_{\vec{m}}}/\mathbb{G}_m, A_{\vec{m}}) \times_{\mathfrak{pgl}_{n}\sslash \GL_{n}} \M_C(\GL_{A_{n}}/\mathbb{G}_m, A_{n}))/\M_C(\PGL_{n}, 0)\end{equation} of affine varieties. Combining the isomorphisms of \cref{lemma that relevant Coulomb branches are Kostant sections} and the isomorphism \cref{BFM Iso} we see that \cref{Coulomb Branch for Exploded and Glued Quivers} gives rise to an isomorphism \begin{equation}\label{Coulomb Branch for Exploded and Glued Quivers After Langlands Duality}\M_C(\GL_{Q_{\vec{m}}}/\mathbb{G}_m, Q_{\vec{m}}) \cong ( \Kost{S(L)}\times_{\mathfrak{c}_{\SL_{n}}}\Kost{\SL_n})/J_{SL_{n}}\end{equation} but the main result of \cite{GannonCotangentBundleofParabolicBaseAffineSpaceandKostantWhittakerDescent} shows that the right hand side of \cref{Coulomb Branch for Exploded and Glued Quivers After Langlands Duality} is isomorphic to $\overline{T^*(\SL_{n}/U_P)}$, as desired. 
\end{proof}

\subsection{Mirror Sicilian theories}\label{Mirror Sicilian Theories Subsection}
One natural setting where we can apply the above results are the examples provided by \textit{mirror Sicilian theories}, as studied by Benini, Tachikawa and Xie in \cite{beniniMirrors3d2010}.  These are theories associated to a choice of group $G$ and a punctured Riemann surfaces of genus $g$ with $p$ punctures, together with a choice of defect at each puncture.  These defects correspond to the Nahm pole boundary condition for a choice of $\operatorname{SL}_2$ mapping to the group $G$, up to conjugacy; these in turn are classified by the nilpotent orbits in $\LG$.

\subsubsection{Mirror Sicilian theories for $\GL_n$}In the case $G=GL_n$, Sicilian theories have been constructed  as the mirror of quiver gauge theories for {\bf comet-shaped quivers}, obtained by joining together $g$ loops and $p$ legs at a single vertex. In other words, comet-shaped quivers are those quivers whose underlying graph has the form
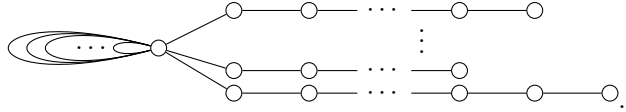
\begin{equation*}
\begin{tikzpicture}
    \tikzstyle{gaugered} = [circle, draw=black, inner sep=1pt, minimum size=6pt]
    \tikzstyle{gaugeblue} = [circle, draw=black, inner sep=1pt, minimum size=6pt]
    \tikzstyle{gaugepurple} = [circle, draw=black, inner sep=1pt, minimum size=6pt]
    \tikzstyle{gauge} = [circle, draw, inner sep=1pt, minimum size=6pt]

    \node[gaugered] (5) at (5,0) {};
    \draw (5) to[loop left, looseness=30] (5);
    \path (5) ++(-0.9,0) node {$\hdots$}; 
    \draw (5) to[loop left, looseness=85] (5);
    \draw (5) to[loop left, looseness=100] (5);
    \draw (5) to[loop left, looseness=115] (5);

    \node[gaugeblue] (6u) at (6,0.5) {};
    \node[gaugeblue] (6u2) at (7,0.5) {};
    \node (7u) at (8,0.5) {$\cdots$};
    
    \node[gaugeblue] (9u) at (10,0.5) {};
    \node[gaugeblue] (8u) at (9,0.5) {};
    \node[gaugeblue] (6m) at (6,-0.3) {};
    \node[gaugeblue] (6m2) at (7,-0.3) {};
    \node (7m) at (8,-0.3) {$\cdots$};
    \node (7x) at (8.5, 0.2) {$\vdots$};
    \node[gaugeblue] (8m) at (9,-0.3) {};
    \node[gaugeblue] (6d) at (6,-0.6) {};
    \node[gaugeblue] (6d2) at (7,-0.6) {};
    \node (7d) at (8,-0.6) {$\cdots$};
    \node[gaugeblue] (8d) at (9,-0.6) {};
    \node[gaugeblue] (9d) at (10,-0.6) {};
    \node[gaugeblue] (10d) at (11,-0.6) {};
    
    \draw (5)--(6u)--(6u2)--(7u)--(8u)--(9u)
          (5)--(6m)--(6m2)--(7m)--(8m)
          (5)--(6d)--(6d2)--(7d)--(8d)--(9d)--(10d);
\end{tikzpicture}.
\end{equation*}
For each puncture, we have a choice of partition $h_1\geq \cdots \geq h_k$ of $n$, given by the Jordan type of the corresponding nilpotent orbit.  This describes the dimension vector along the corresponding leg as follows: on the vertex $m$ steps away from the central vertex, the dimension vector is $n_m=h_{m+1}+\cdots +h_k$.  That is, these dimension vectors decrease as we move away from the central node, and this decrease is concave ($2n_m\leq n_{m-1}+n_{m+1}$). In particular, if the nilpotent is regular, we have a single node $(n)$, and if it is trivial, then we obtain $(n,n-1,\dots,1)$.  
\begin{Remark}
	It's very confusing to straighten out all the appearances of Langlands duality in the literature on these theories.  To help the reader interested in these details, let us note:
	\begin{enumerate}
		\item We are using Nahm pole boundary conditions corresponding to nilpotents $e\in \mathfrak{g}$.   
		\item We can then use the Janus boundary condition and $S$-duality to construct the theory $T_{e}(G^{\vee})$ and its $S$-dual $T^{e}(G)$ as in \cite[\S 4.2]{gaiottoSDuality2009}.  The intersection with the nilpotent cone of the Slodowy slice to the nilpotent $e$  can be interpreted as the Coulomb and Higgs branches of these theories, respectively.
		\item The theory $T^{e}(G)$ corresponds to the defect at each puncture;  in particular, if $e=0$, then $T^{e}(G)=T(G)$, and this gives a ``maximal puncture'' in the phrasing of \cite{beniniMirrors3d2010}.
		\item As discussed in \cite[Rem. 5.3]{BravermanFinkelbergNakajimaRingObjectsIntheEquivariantDerivedSatakeCategoryArisingFromCoulombBranches}, the Higgs branch of the Sicilian theory will then be the Coulomb branch of a gauge theory for $G^{\vee}$, an expectation we make more precise in \cref{Conjecture for Gluing Everything}.
	\end{enumerate}
	
\end{Remark}   

Of course, the quiver associated to one of these Sicilian theories has a natural dismemberment where we separate off the loops at the central vertex and then consider each leg separately. Such a quiver satisfies the conditions of \cref{lem:gluable quiver} if $g=0$. In the case of $G = \SL_n$ or $G = \GL_n$ with nilpotents for all punctures regular, one can use \cref{gluing for quiver case} together with \cref{lemma that relevant Coulomb branches are Kostant sections} to give an alternate proof of the fact that the \textit{Moore-Tachikawa varieties} are the Coulomb branches of the associated star-shaped quiver, originally proved in \cite[Section 5]{BravermanFinkelbergNakajimaRingObjectsIntheEquivariantDerivedSatakeCategoryArisingFromCoulombBranches}. 

More generally, in the genus 0 case, the dismemberment corresponding to removing a single leg gives an operation corresponding to adding a single puncture to a Riemann surface, mentioned as a desideratum in \cite[\S 4.7]{bravermanfinkelbergsurvey}. This uses as an input the Coulomb branch $\M_{\mathbf{h}}=\M(G_Q,\mathbf{n} )$ for a single leg $Q$ with dimension vector $\mathbf{n}$ as above. In fact, we can describe the space $\M_{\mathbf{h}}$ as a bow variety in the more general setting where $Q$ is a type $A_{s + r + 1}$ quiver for which the dimension vector $(n_{-s},\dots, n_0,\dots, n_r)$ satisfies \begin{equation}\label{Inequalities for n}n_{-s}<\cdots < n_{-1}<n=n_0> n_1> \cdots > n_{r}\end{equation} using work of Nakajima--Takayama \cite[Theorem 6.18]{nakajimaCherkisBow2017}.  

To make this description more explicit, we first recall that a {\bf companion matrix} for a polynomial $p$ with $\deg p=n$ is the matrix for the action of $u$ on $\C[u]/(p(u))$ in the basis $\{1,u,\dots,u^{n-1}\}$.  That is, it is a matrix of the form
\[C_p=\begin{bmatrix}
    0 & 0 & 0 & \cdots & -p_0\\
    1 & 0 & 0 & \cdots & -p_1\\
    0& 1 & 0 &\cdots & -p_2\\
    \vdots &\vdots &\vdots & \ddots &\vdots \\
    0 & 0 & 0& \cdots & -p_{n-1}
\end{bmatrix}\]

For any sequence of positive integers $x_1,\dots, x_r$, let $\mathcal{S}_{x_1,\dots, x_r}$ be the set of matrices where \begin{enumerate}
    \item The block diagonal with sizes $x_1,\dots, x_r$ is given by arbitrary companion matrices.
    \item Below the block diagonal, all non-zero entries are in the $y_k$th rows where $y_k=x_1+\cdots + x_k+1$.
    \item Above the block diagonal, all non-zero entries are in the $z_k$th columns, where $z_k=x_1+\cdots +x_{k+1}$.
\end{enumerate}
That is, when we divide $A$ into blocks with sizes $x_1,\dots, x_r$, they have the form
\[
\left[\begin{array}{{c|c|c|c}}
    C_{p_1} & D_{12} & \cdots & D_{1r}\\ \hline
    E_{21} & C_{p_2} & \cdots & D_{2r}\\\hline
    \vdots & \vdots & \ddots & \vdots\\\hline
    E_{r1} & E_{r2} & \cdots &C_{p_r}
\end{array}\right]
\]
where all $D_{ij}$'s only have non-zero last column, and $E_{ij}$'s only have non-zero top row.  
This space appears in \cite[E.5]{nakajimaCherkisBow2017}\footnote{This erratum only appears in the \href{https://arxiv.org/pdf/1606.02002v4}{version 4 on the arXiv}, not in the published version.}, where $GL(n)\times \mathcal{S}_{x_1,\dots, x_r}$ is interpreted as a space of solutions to Nahm's equations.  

As noted in \cite[E.5]{nakajimaCherkisBow2017}, if the sequence $x_1,\dots, x_n$ is weakly increasing, then $\mathcal{S}_{x_1,\dots, x_r}$ is a normal slice to the nilpotent with (lower triangular) Jordan type $x_1,\dots, x_r$ (which is the minimal rank element of $\mathcal{S}_{x_1,\dots, x_r}$) by \cite[Lemma 3.2.2]{MV22}. 

Let \[\gamma(\mathbf{n})=\{(X,g)\mid X\in \mathcal{S}_{n_{r}, n_{r+1}-n_{r},\dots , n_0-n_{1}}, g^{-1}Xg\in \mathcal{S}_{n_{-s}, n_{-s+1}-n_{-s}, \dots, n_0-n_{-1}  }^{T}\}.\]  This variety is equipped with a map $\gamma(\mathbf{n})\to \fc{\GL_{n_k}}$ for $k\geq 0$ is given by taking the orbit of the top left $n_k\times n_k$ minor $X_k$ of $X$, and for $k\leq 0$ by the orbit of the top left $n_k\times n_k$ minor $Y_k$ of $Y=gXg^{-1}$.  

The schemes $\gamma(\mathbf{n})$ also carry an action of $\M(\GL_{n_k})$. In order to describe this action, we can identify $\M(\GL_{n_k})$ with the set of pairs \[\M(\GL_{n_k})=\{(P,Q) \mid P(u) \text{ monic of degree } n_k, Q(u)\in \C[u]/P[u] \text{ a unit}\}.\]  Let $\iota_k\colon \GL_{n_k}\to \GL_{n}$ be the usual inclusion as the upper left corner and $Z_k=\iota_k(Q(X_k))$ for $k\geq 0$ and $Z_k=\iota_k(Q(Y_k))$ if $k<0$.  Then we can define an action of $\M(\GL_{n_k})$ on $\gamma(\mathbf{n})$ as an $\fc{\GL_{n_k}}$ scheme via the formula:
\begin{equation}\label{eq:bow-action}
    (P,Q)\cdot (X,g)=\begin{cases}
    (Z_kXZ_k^{-1},Z_kg) & k>0\\
    (Z_kXZ_k^{-1}, Z_kg) & k=0\\
    (X,gZ_k) & k<0.
\end{cases} \qquad P=\begin{cases}
    \det(uI-X_k)& k\geq 0\\
    \det(uI-Y_k)& k\leq 0
\end{cases}
\end{equation}
\begin{Lemma}\label{lem:bow}
If $Q$ is a type $A_{r + s + 1}$ quiver for some nonnegative integers $r, s$ for which the dimension vector $(n_{-s},\dots, n_0,\dots, n_r)$ satisfies \cref{Inequalities for n}, there is a $\M(\GL(V_Q))$-equivariant isomorphism of the Coulomb branch $M(G_Q,\mathbf{n} )$ with the subset $\gamma(\mathbf{n})\subset \mathfrak{gl}_n\times \GL_n $.
\end{Lemma}
It's worth noting that this result follows naturally from the description of bow varieties in \cite[\S 3]{nakajimaCherkisBow2017} and \cite[\S 2]{takayamaNahmsEquations2016} and, in particular, is a direct application of the ``divide in the middle'' trick used to prove \cite[Theorem 7.11]{nakajimaCherkisBow2017} and in \cite[Theorem 2.10 \& 3.6]{jiBowVarieties2023}), but it seems from our discussions with experts that this version has not appeared in the literature.  

\begin{proof}[Proof of \cref{lem:bow}]
By \cite[Theorem 6.18]{nakajimaCherkisBow2017}, the Coulomb branch of a quiver $Q$ of type $A$ is the moduli space of representations of the handsaw quiver 
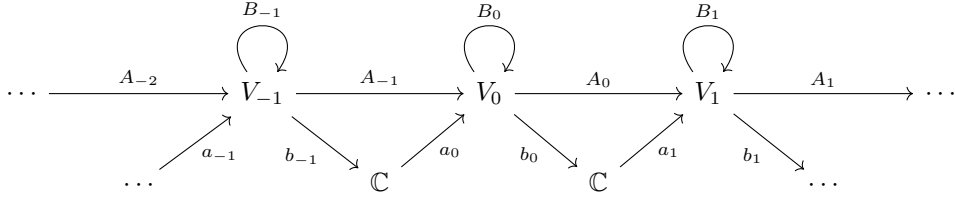
\[\begin{tikzcd}
	\cdots && {V_{-1}} && {V_0} && {V_1} && \cdots \\
	& \cdots && {\mathbb{C}} && {\mathbb C} && \cdots
	\arrow["{A_{-2}}", from=1-1, to=1-3]
	\arrow["{B_{-1}}", from=1-3, to=1-3, loop, in=55, out=125, distance=10mm]
	\arrow["{A_{-1}}", from=1-3, to=1-5]
	\arrow["{b_{-1}}"', from=1-3, to=2-4]
	\arrow["{B_0}", from=1-5, to=1-5, loop, in=55, out=125, distance=10mm]
	\arrow["{A_0}", from=1-5, to=1-7]
	\arrow["{b_0}"', from=1-5, to=2-6]
	\arrow["{B_1}", from=1-7, to=1-7, loop, in=55, out=125, distance=10mm]
	\arrow["{A_1}", from=1-7, to=1-9]
	\arrow["{b_1}"', from=1-7, to=2-8]
	\arrow["{a_{-1}}"', from=2-2, to=1-3]
	\arrow["{a_0}"', from=2-4, to=1-5]
	\arrow["{a_1}"', from=2-6, to=1-7]
\end{tikzcd}\]
where $\dim(V_i)=n_i$ satisfying the equations
\begin{equation}\label{handsaw-mm}
    A_iB_{i}-B_{i+1}A_i-a_{i+1}b_i=0
\end{equation}
and stability conditions
\begin{itemize}
    \item[(S1)] no non-zero subspace $S_i\subset V_i$ satisfies $B_i(S_i)\subset S_i, A_i(S_i)\subset S_{i+1}, b_i(S_i)=0$.
    \item[(S2)] no proper subspace $T_i\subset V_i$ satisfies $B_i(T_i)\subset T_i, A_{i}(T_{i})=T_{i+1}, a_i(1)\subset T_i$.
\end{itemize}

Thus, we need to produce an isomorphism of $\gamma(\mathbf{n})$ with the representations of the handsaw quiver in the unimodal case.
Given $(X,g)\in  \gamma(\mathbf{n})$ with $Y=g^{-1}Xg$, we can construct a representation of the handsaw quiver as follows.
\begin{enumerate}\renewcommand{\theenumi}{\alph{enumi}}
	\item If $k\geq 0$, let the map $A_k\colon \C^{n_k}\to \C^{n_{k+1}}$ be projection to the first $n_{k+1}$ coordinate vectors, $B_k$ be the top $n_{k}\times n_k$ corner $X_k$ of $X$, $b_k$ the projection to the last coordinate, and $a_{k+1}$ the top $n_{k+1}$ entries of the last column of $X_k$.  
	\item If $k\leq -1$, let $A_{k-1}\colon \C^{n_{k-1}}\to \C^{n_{k}}$ be the inclusion of the first $n_k$ coordinate vectors, $B_k$ be the top $n_{k}\times n_k$ corner $Y_k$ of $Y$, $a_k$ the unit vector $(0,\dots, 1)$ and $b_k$ the dot product with the leftmost $n_k$ entries of the last row of $Y_{k+1}$
	\item Let $A_{-1}$ be the first $n_{-1}$ columns of $g$ and $a_0$ multiplication by the last column of $g$.
\end{enumerate}
We can describe this a bit more systematically by performing construction (a) using $X$ to obtain the representation for $k\geq 0$,  then construction (a) again for $Y^T$ to obtain the dual of the representation for $k\leq 0$.  We then glue the vector spaces $\C^{n_0}$ that appear in both constructions using $g$ for the change of basis.  

This always results in a stable representation by \cite[Prop. 2.14]{takayamaNahmsEquations2016}.  Thus, we have a map $\gamma(\mathbf{n})\to M(G_Q,\mathbf{n} )$.  To show the isomorphism, we construct its inverse.  

We do this by inductively constructing the corresponding basis on the vector spaces $V_k$.
By \cite[Lemma 2.19]{takayamaNahmsEquations2016},
if we consider the filtration \[U_{k,m}=\{v\in V_k \mid b_kv=b_kB_kv=\cdots = B_k^{m-1}v=0\},\] then $U_{k,n_k-n_{k+1}}$ is a complement to $\ker A_k$.   In particular, given a basis $v_{k+1,1},\dots, v_{k+1,n_{k+1}}$ of $V_{k+1}$, we can uniquely lift this to a basis $v_{k,1},\dots, v_{k,n_{k+1}}$ of $U_{k,m}$.  Furthermore, there is a unique extension of this to a basis of $U_{k,m-1}$ by letting $v_{k,n_{k+1}+1}$ be the unique element of $U_{k,m-1}\cap \ker A_k$  such that $b_k(B_k^{n_k-n_{k-1}}v_{k,n_{k+1}+1}=1$.  We can extend this to a basis of the whole space by $v_{k,n_{k+1}+j}=B_{k}^{j-1}v_{k,n_{k+1}+1}$.  
In this basis, the endomorphism $B_k$ acts in Hurtubise normal form, that is:
\begin{equation*}
    B_{k}^{(v)}=\begin{bmatrix}
B_{k+1}^{(v)} & D\\  
   E     & C_p
    \end{bmatrix} 
\end{equation*}
where only the first row of $E$ and last column of $D$ are non-zero and $C_p$ is a companion matrix.  
  Given a stable representation of the handsaw quiver, let $X=B_{0}^{(v)}$ denote the matrix of $B_0$ in the basis we have constructed.

On the spaces for non-positive indices, we can consider the spaces $V_0^*, V_{-1}^*,\dots, V_{-s}^*$, the transposes of $A_k,B_k,a_k,b_k$ and apply the same argument again.  This produces a basis $w_{0,1}^*,\cdots,w_{0,n_0}^*$  of $V_0^*$ such that $B_0^{(w)}\in \mathcal{S}_{n_0-n_{-1},\dots, n_{-s+1}-n_{-s}, n_{-s}}^T$.  Let $g$ be the unique invertible matrix given by $g^{-1}_{ij}=w_i^*(v_j)$, which relates these bases.  The map sending a handsaw representation to $(X,g)$ is our desired inverse isomorphism $  M(G_Q,\mathbf{n} )\to \gamma(\mathbf{n})$. 
This completes the proof of the isomorphism.

Now, we turn to the action of $\M(\GL(V_Q))$ on the bow variety.  By the factorization property \cite[Prop 3.6]{nakajimaCherkisBow2017}, it suffices to prove that the formulas \cref{eq:bow-action} are correct for the case where $V_0$ is 1-dimensional and other $V_i$ are 0.  In this case, $B_0$ acts by a complex scalar, which we denote by the same symbol, and $Q$ must be a non-zero constant. This induces an isomorphism to $\C\times \C^{\times}$.  On the other hand, our isomorphism gives the $1\times 1$ matrix $X=(B_0)$, and $g=b_0a_0$.   Thus, the action described multiplies $g$ by the complex number $Q$ as indicated in \cref{eq:bow-action}.  
\end{proof}

Note that \cref{lemma that relevant Coulomb branches are Kostant sections} shows that if $h_1=\cdots =h_n=1$, then $\M_{\mathbf{h}}\cong \Kost{\GL_n}\cong \fc{\GL_n}\times GL_n$.   This space still carries a Hamiltonian $GL_n$ action by right multiplication; the moment map, and thus the action, can be easily constructed from the Coulomb perspective as well, using the formulas \cite[(24-26)]{WeekesCoulomb}. This action is frequently used in the physics literature: it is the ``symmetry enhancement'' arising from the fact that all but one of the nodes of the quiver are ``balanced'' (see, for example, \cite[\S 7]{DancerHananyKirwanSymplecticDualityandImplosions}).
\begin{Theorem}\label{mirror Sicilian}
If $\mathcal{T}_0$ is a mirror Sicilian theory for $GL_n$ with $g=0$, and $\mathcal{T}_1$ is the same theory with a puncture for the nilpotent with Jordan type $\mathbf{h}$ added, then 
\[\M_C(\mathcal{T}_1)\cong \M_C(\mathcal{T}_0)\times_{\fc{\GL_n}}^{\M(\GL_n)}\M_{\mathbf{h}}.\]
If $h_1=\cdots =h_n=1$, then 
\[\M_C(\mathcal{T}_1)\cong \M_C(\mathcal{T}_0)\times^{\M(\GL_n)}GL_n\] where we view $\M_C(\mathcal{T}_0)\times GL_n$ as a $\fc{\GL_n}$-space by projection to the first factor and each fiber of $\M(\GL_n)$ acts on $\GL_n$ by left multiplication of the regular centralizer. 
\end{Theorem}

Quotienting by $\M(\GL_n)$ allows a more flexible calculus than \cite{beniniMirrors3d2010}, where they focus on building Riemann surfaces up from pairs of pants, interpreting the gluing of two punctures with regular label as the symplectic quotient by $GL_n$. We also note that the spaces appearing in \cref{mirror Sicilian} naturally obtain a right action by $\GL_n$ through right multiplication.

\subsubsection{Conjectural Generalization for Arbitrary $G$}For general $G$, an approach to studying mirror Sicilian theories was described by  Moore--Tachikawa \cite{MooreTachikawaOn2dTQFTsWhoseValuesareHolomorphicSymplecticVarieties}.  We can interpret the balanced product over $\M(G^{\vee})$ as the effect on the Higgs branches of the connected sum of surfaces.  
By \cite[(3.10)]{MooreTachikawaOn2dTQFTsWhoseValuesareHolomorphicSymplecticVarieties}, the space $\Kost{G}$ corresponds to a disk in the Moore-Tachikawa TQFT.  While capping off a boundary component corresponds to the symplectic quotient $(\Kost{G}\times X) /\!\!/G$, using the right $G$-action on $\Kost{G}$ and the $G$-action on $X$ corresponding to the boundary component, we can instead view $\Kost{G}\times_{\fc{G^{\vee}}}^{\M(G^{\vee})}X$ as connected sum with a disk, i.e. cutting out a disk and creating a new boundary component.  

This description is readily generalized to a conjecture where we ignore the gluability hypothesis, consider more general dimension vectors, and consider an arbitrary gauge group $G$:

\begin{Conjecture}\label{Conjecture for Gluing Everything}\hfill 
\begin{enumerate}
    \item 
    The Coulomb branch of the mirror Sicilian theory of an unpunctured Riemann surface of genus $g$ is 
    \[\M_{g-\operatorname{loop}}^{(G)}=\M(G^{\vee}, (\mathfrak{g}^{\vee})^{\oplus g})\cong \M(T^{\vee}, (\mathfrak{g}^{\vee})^{\oplus g-1})/W.\]
    The $\M(G^{\vee})$-action on $\M_{g-\operatorname{loop}}^{(G)}$ is induced by the map $\M(G^{\vee})_{\LT^{\vee}}\to \M(T^{\vee})\cong \LT^{\vee}\times T$, which acts on $\M(T^{\vee}, (\mathfrak{g}^{\vee})^{\oplus g-1})$ by the usual $T$-action.   
    \item The Coulomb branch $\M_e$ of the mirror Sicilian theory of a Riemann surface of genus $0$ with a single puncture with label $e$ is written as the preimage \[\M_e=\mu^{-1}(\mathcal{S}_{\mathsf{reg}}\times \mathcal{S}_e)=\{(X,g)\in \mathcal{S}_{\mathsf{reg}}\times G |-\operatorname{ad}_{g^{-1}}X\in \mathcal{S}_e \}\] under the moment map $\mu\colon T^*G\to \mathfrak{g}^*\times \mathfrak{g}^*$ for the left and right action of $G$ of $\mathcal{S}_{\mathsf{reg}}\times \mathcal{S}_e$ where $\mathcal{S}_e$ is the Slodowy slice for a nilpotent $e$.  Under the identification of $\M(G^{\vee})$ with the regular centralizers of $G$, the action on $\M_e$ is by left multiplication on $g$.
    \item The Coulomb branch of the mirror Sicilian theory for $G$ on a Riemann surface of genus $g$ with punctures labeled by nilpotents $e^{(1)},\dots, e^{(p)}$, that is, the Higgs branch for the Sicilian theory with the same data, is isomorphic to the quotient of
    \[\M_{g-\operatorname{loop}}^{(G)}\times_{\fc{G^{\vee}}}\M_{e^{(1)}}\times_{\fc{G^{\vee}}}\cdots \times_{\fc{G^{\vee}}}\M_{e^{(p)}} \cong \M_{g-\operatorname{loop}}^{(G)}\times_{\fc{G^{\vee}}} \{(X,g_1,\dots, g_p)\in \mathcal{S}_{\mathsf{reg}}\times G^p\mid \operatorname{ad}_{g_i^{-1}}X\in \mathcal{S}_{e^{(i)}}\}\]
    by the natural action of \[\M(G^{\vee})^{p+1}_{0}:=\{(a_0,\dots, a_{p})\in \M(G^{\vee})\times_{\fc{G^{\vee}}}\cdots  \times_{\fc{G^{\vee}}}\M(G^{\vee}) \mid a_0\cdots a_{p}=1\}.\]
\end{enumerate}
\end{Conjecture}

\begin{Remark} The proposals for the spaces assigned to basic pieces have already appeared in the literature, so the new part of \cref{Conjecture for Gluing Everything} is how to assemble these using connect sum. We now specifically highlight instances of \cref{Conjecture for Gluing Everything} which have already appeared in the literature, both physically and mathematically:
\begin{enumerate}
    \item We first note that the Coulomb branch for a quiver with one vertex and $g$ loops can be computed from
\cite[Theorem 3]{TelemanCoulombBranchesforQuaternionicRepresentations}: the result is the symmetrized abelian Coulomb branch \[\M_{g-\operatorname{loop}}^{(n)}=\M(Q,n_i)=\M(\Gm^{n_i}, \End(\mathbb C^{n_i})^{\oplus g-1})/S_n.\]  The important special special case of $g=1$ is shown in \cite[Proposition 5.24]{BravermanFinkelbergNakajimaRingObjectsIntheEquivariantDerivedSatakeCategoryArisingFromCoulombBranches}, where this quotient is $\M(Q,n_i)=\Sym^n(T^*\Gm)$. Thus, \cref{Conjecture for Gluing Everything}(1) holds for $G = \GL_n$.
    \item \cref{Conjecture for Gluing Everything}(1) can be interpreted by a special case of mirror symmetry for the theory $T[G]$; see \cite[Remark 5.3]{BravermanFinkelbergNakajimaRingObjectsIntheEquivariantDerivedSatakeCategoryArisingFromCoulombBranches}  (note that we have switched $G$ and $G^{\vee}$ from their notation) and the references therein.
    \item \cref{Conjecture for Gluing Everything}(2) is easily generalized to the case of two punctures with labels $e_1,e_2$, where we consider $\mu^{-1}(\mathcal{S}_{e_1}\times \mathcal{S}_{e_2})$. This is already predicted in \cite[Figure 51(c)]{gaiottoSDuality2009}.
    \item In light of \cref{Intro Statement of Functoriality Theorem}, we can interpret \cref{Conjecture for Gluing Everything} as describing the Higgs branch of the Sicilian theory as the Coulomb branch of the theory $\operatorname{Hyp}((\mathfrak{g}^{\vee})^{\oplus g})\times T_{e_1}(G^{\vee}) \times \cdots \times T_{e_n}(G^{\vee})$ with gauge group restricted to $G^{\vee}$.  It's worth noting that the Higgs branch of this theory for $G=\GL_n$ is identified with the Coulomb branch of the Sicilian theory in \cite[(A.11)]{beniniMirrors3d2010}, so this gives the dual identification suggested in {it loc.\ cit.}
    \end{enumerate}

\end{Remark}

\subsection{Particular examples}For some partitions of $n$, the Coulomb branch associated to the quiver $Q_{\vec{m}}$ has a known description which is not a priori related to $\overline{T^*(G/U_P)}$. We highlight some cases of interest here. 

\subsubsection{The Moore-Tachikawa variety for $\SL_3$} In the special case that our partition is $4 = 2 + 2$ (originally suggested to us by Amihay Hanany) we obtain the following description of $\overline{T^*(\SL_4/U_P)}$. In this case, the quiver $Q_{\vec{m}}$ gives the quiver for the \textit{Moore-Tachikawa variety} for $\SL_3$, see \cite[Section 5]{BravermanFinkelbergNakajimaRingObjectsIntheEquivariantDerivedSatakeCategoryArisingFromCoulombBranches}, and so \cref{Full Statement Coulomb Branches of Particular Quivers} implies that the Moore-Tachikawa variety for $\SL_3$ is isomorphic to $\overline{T^*(\SL_4/U_{(2, 2)})}$ where \[U_{(2, 2)} := \begin{pmatrix}1 & 0 & \ast & \ast \\
0 & 1 & \ast & \ast \\ 
0 & 0 & 1 & 0 \\ 0 & 0 & 0 & 1\end{pmatrix}\] is the unipotent radical of the standard parabolic subgroup of $\SL_4$ associated to this partition. On the other hand, the Moore-Tachikawa variety for $\SL_3$ is also known to be isomorphic to the closure $\mathcal{N}_{\mathrm{min}}$ of the minimal nonzero nilpotent orbit of $\mathfrak{e}_6$, see \cite[Section 5(iv)]{BravermanFinkelbergNakajimaRingObjectsIntheEquivariantDerivedSatakeCategoryArisingFromCoulombBranches}. Therefore, we obtain the following result, which can be viewed as a \lq parabolic analogue\rq{} of \cite[Theorem 5.2]{JiaTheGeometryOfTheAffineClosureOfCotangentBundleOfBasicAffineSpaceForSLn}:

\begin{Corollary}\label{Exceptional Iso for Moore Tachikawa Variety}
There is an isomorphism $\mathcal{N}_{\mathrm{min}} \cong \overline{T^*(\SL_4/U_{(2, 2)})}$ of affine varieties.
\end{Corollary}

Shortly after the first version of this document was posted, an alternative and more direct proof of \cref{Exceptional Iso for Moore Tachikawa Variety} was given in \cite{JiaMinimalNilpotentOrbitsOfTypeDandE}.

\subsubsection{Based maps isomorphisms} In the case that our partition is $n := 3 = 2 + 1$, respectively $n := 4 = 3 + 1$, the underlying graph of our quiver $Q_{\vec{m}}$ can be identified with the Dynkin diagram of type $D_{10}$, respectively $E_7$. In this case, \[U_{(2, 1)} := \begin{pmatrix}1 & 0  & \ast \\
0 & 1 & \ast \\ 
0 & 0 & 1\end{pmatrix}\text{, respectively }U_{(3, 1)} := \begin{pmatrix}1 & 0 & 0 & \ast \\
0 & 1 & 0 & \ast \\ 
0 & 0 & 1 & \ast \\ 0 & 0 & 0 & 1\end{pmatrix}\] gives the unipotent radical of the associated standard parabolic subgroup of $\GL_3$, respectively $\GL_4$. By \cite[Theorem 3.1]{BravermanFinkelbergNakajimaQuiver}, the Coulomb branch of the respective quiver can be identified with moduli space $\mathring{Z}_G^{\gamma}$ of based rational maps from $\mathbb{P}^1$ to the flag variety of the simple group $G$ whose associated Dynkin diagram is type $D_{10}$, respectively $E_7$, of some fixed degree $\gamma$. More precisely, letting $\alpha$, respectively $\beta$, denote the sum of simple roots with multiplicites labeled by $Q_{\vec{m}}$, using \cref{final main theorem}, we obtain the following corollary:

\begin{Corollary}
There are isomorphisms \[\overline{T^*(\GL_3/U_{(2,1)})} \cong \mathring{Z}^{\alpha}_{\mathrm{SO}(10)} \text{ and } \overline{T^*(\GL_4/U_{(3,1)})} \cong \mathring{Z}_{E_7}^{\beta}\] of affine varieties.
\end{Corollary}

\appendix
\section{Connection to geometric Satake}\label{Appendix to Avoid Derived Satake in Main Body}

In this appendix, we discuss the connection of our results to the  \textit{derived geometric Satake equivalence} \cite{BezrukavnikovFinkelbergEquivariantSatakeCategoryandKostantWhittakerReduction} and \textit{relative Langlands duality} \cite{BenZviSakleredisVenkateshRelativeLanglandsDuality}. 

\subsection{Background on geometric Satake}

We let $\sphericalHeckeCatForTHISGROUP{G}$ denote the spherical Hecke category
for $G$ (as in, for example, \cite[Section 6.6]{BenZviSakleredisVenkateshRelativeLanglandsDuality}) and assume that $\mathcal{A} := \AGN$ is the algebra object of $\sphericalHeckeCatForTHISGROUP{G}$, constructed in \cite[\S 2(iv)]{BravermanFinkelbergNakajimaRingObjectsIntheEquivariantDerivedSatakeCategoryArisingFromCoulombBranches}, whose cohomology in the abelian category $H^*_G(\mathrm{pt})$-mod gives the Coulomb branch as a $H^*_G(\mathrm{pt})$-algebra. Observe that, since the affine Grassmannian $\Gr_G$ for $G$ is ind-proper, we have a canonical map $t_{*}(\mathcal{A}) \to t_*t^!t_*(\mathcal{A})$ where $t$ is the terminal map; taking cohomology, we obtain a graded map \begin{equation}\label{Explicit Coaction as Map of Graded Vector Spaces}
H^*_G(\mathcal{A}) \xrightarrow{} H^*_G(t_*t^!t_*(\mathcal{A})) \cong \mathcal{O}(\M_C(G, 0)) \otimes_{H_G^*(\mathrm{pt})} H^*_G(\mathcal{A})
\end{equation} which one can readily check equips $H^*_G(\mathcal{A})$ with the structure of a graded comodule for $\mathcal{O}(\M_C(G, 0))$. 

In particular, this construction gives an action of $\mathcal{M}_C(G, 0)$ on $\M_C(G, \bfN)$ and, since this action extends the usual group multiplication on $\M_C(G, 0)$, we see that the uniqueness assertion of \cref{Action on General Coulomb Branch by Gluing} implies that this action agrees with the action of $\M_C(G, 0)$ on $\M_C(G, \mathbf{N})$ constructed in \cref{Action on General Coulomb Branch by Gluing}. Letting $\semisimpleComplexesOnGr$ denote the full subcategory of objects of $\overline{\mathcal{H}_G}$ which are (possibly infinite) direct sums of IC sheaves, one can similarly equip the cohomology $H^*_{G_{\mathcal{O}}}(C)$ of any $C \in \semisimpleComplexesOnGr$ with the structure of a graded comodule for $\mathcal{O}(\M_C(G, 0))$. Therefore, we obtain a functor \[\semisimpleComplexesOnGr \xrightarrow{H^*_{G_{\mathcal{O}}}(-)} H_{BM}^{G_{\mathcal{O}}}(\Gr_G)\mathrm{-comod}^{\mathbb{G}_m}\] to the category of graded representations of $\M_C(G, 0)$. Recall that there is a fully faithful functor \begin{equation}\label{Fully Faithful Inclusion for BM Comodules}\mathfrak{o}_{BM}: H_{BM}^{G_{\mathcal{O}}}(\Gr_G)\mathrm{-comod}^{\mathbb{G}_m} \subseteq H^*_{G_{\mathcal{O}}}(\Gr_G)\mathrm{-mod}^{\mathbb{G}_m}\end{equation} which is the identity on objects and morphisms; see \cite[Theorem 4.3]{BrzezinskiWisbauerCoringsandComodules}. The arguments of \cite[Proposition 2.7]{YunZhuIntegralHomologyofLoopGroupsviaLanglandsDualGroups} (see also \cite[Section 3.5]{BezrukavnikovFinkelbergEquivariantSatakeCategoryandKostantWhittakerReduction}) show that the composite of the functor $H^*_{G_{\mathcal{O}}}(-)$ with the functor \cref{Fully Faithful Inclusion for BM Comodules} is monoidal. Since \cref{Fully Faithful Inclusion for BM Comodules} is fully faithful, we obtain a monoidal structure on the functor $H^*_{G_{\mathcal{O}}}(-)$ itself.


\newcommand{\LGvd}{\mathfrak{g}^*}
As above, we let $G^{\vee}$ denote the Langlands dual group of $G$ and denote by $\LGc$ its Lie algebra. Let $\QCoh(\LGcd)^{G^{\vee} \times \mathbb{G}_m}_{\mathrm{fr}}$ denote the full subcategory of the category $\QCoh(\LGcd)^{G^{\vee} \times \mathbb{G}_m}$ of graded $G^{\vee}$-representations equipped with a compatible $\Sym(\LGc)$-module structure which are isomorphic to objects of the form $\Sym(\LGc) \otimes V$ for some possibly infinite-dimensional representation $V$ of $G^{\vee}$. The group $G^{\vee}$ canonically acquires a pinning via the Tannakian formalism; we view $\cG := \LT\sslash W\cong \LGc^*\sslash G^{\vee}$ as a closed subscheme of $\LGc^*$ as the Kostant section associated to this pinning via the closed embedding $\iota_{G^{\vee}}$. Note that $\iota_{G^{\vee}}(t^2X)=t^2\tilde{\rho}(t)\cdot \iota_{G^{\vee}}(X)$ for $t\in \Gm$ where  $\tilde{\rho}(t)$ is the unique cocharacter to the derived subgroup of $G$ which acts with weight 2 on all simple root spaces.  
Let \[J_{G^{\vee}}=\{(X,g)\in \cG\times G^{\vee}|\operatorname{Ad}^*_g(\iota_{G^{\vee}}X)=\iota_{G^{\vee}}X\}\] denote the centralizer of this Kostant section; this has a natural $\Gm$-action induced by the map $t\cdot (X,g)=(t^2X, \tilde{\rho}(t)g\tilde{\rho}(t)^{-1})$, compatible with the same action on $\cG\times G^{\vee}$.  
The restriction of a $G^{\vee}\times\Gm$-equivariant sheaf to the Kostant slice naturally acquires the structure of a representation of the centralizer $J_{G^{\vee}}$ of this Kostant section with compatible $\Gm$-action, and so from this restriction we obtain a symmetric monoidal functor \[\kappa: \QCoh(\LGc^*)^{G^{\vee} \times \mathbb{G}_m}_{\mathrm{fr}} \to \mathcal{O}(J_{G^{\vee}})\mathrm{-comod}^{\mathbb{G}_m}\] to the category of graded representations of $J_{G^{\vee}}$. Moreover, via the isomorphism \cref{BFM Iso}, we may view $\kappa$ as a functor to the category of graded representations of $\M_C(G, 0)$. 

\subsection{Proof of \cref{lemma that relevant Coulomb branches are Kostant sections}}

With this notation, we may now state the following consequence of the derived geometric Satake equivalence, based on the results of \cite{BezrukavnikovFinkelbergEquivariantSatakeCategoryandKostantWhittakerReduction}:

\begin{Theorem}\label{Satake On Free Objects with Compatibility} There is a monoidal equivalence of categories $\tilde{S}: \QCoh(\LGc^*)^{G^{\vee} \times \mathbb{G}_m}_{\mathrm{fr}} \xrightarrow{\sim} \mathcal{I}$ and a monoidal natural isomorphism $H^*_{G_{\mathcal{O}}}(\tilde{S}(-)) \cong \kappa$. 
\end{Theorem}

\begin{proof}
If we define $\mathcal{O}(J_{G^{\vee}})^* := \Hom_{H_G^*(\mathrm{pt})}(\mathcal{O}(J_{G^{\vee}}), H_G^*(\mathrm{pt}))$, then identical arguments as in the existence of the functor \cref{Explicit Coaction as Map of Graded Vector Spaces} give a fully faithful functor \begin{equation}\label{Fully Faithful Functor B-Side}\mathfrak{o}_{J_{G^{\vee}}}: \mathcal{O}(J_{G^{\vee}})\mathrm{-comod}^{\mathbb{G}_m} \subseteq \mathcal{O}(J_{G^{\vee}})^*\mathrm{-mod}^{\mathbb{G}_m}\end{equation} which is the identity on objects and morphisms. As surveyed in \cite[Section 2.6]{BezrukavnikovFinkelbergEquivariantSatakeCategoryandKostantWhittakerReduction},  there are equivalences of monoidal categories \begin{equation}\label{Directly Stated Equivalencse of Categories by BF}H^*_{G_{\mathcal{O}}}(\Gr_G)\mathrm{-mod}^{\mathbb{G}_m} \cong \QCoh^{\mathbb{G}_m}(T^*(\check{\LT}^*\sslash W)) \cong \mathcal{O}(J_{G^{\vee}})^*\mathrm{-mod}^{\mathbb{G}_m}\end{equation} provided by \cite[Theorem 1]{BezrukavnikovFinkelbergEquivariantSatakeCategoryandKostantWhittakerReduction} and \cite{ginzburg2000perversesheavesloopgroup}, respectively, where $\check{\LT}^*$ is the dual Lie algebra to the maximal torus $T^{\vee}$ given by our pinning. Now, by taking the ind-extension of the monoidal equivalence in \cite[Theorem 4]{BezrukavnikovFinkelbergEquivariantSatakeCategoryandKostantWhittakerReduction}, we obtain a monoidal equivalence $\tilde{S}$ such that, identifying the monoidal categories in \cref{Directly Stated Equivalencse of Categories by BF}, there is a monoidal natural isomorphism \[\mathfrak{o}_{BM}\circ H^*_{G_{\mathcal{O}}}(\tilde{S}(-)) \cong \mathfrak{o}_{J_{G^{\vee}}} \circ \kappa(-).\] However, since the functors $\mathfrak{o}_{BM}$ and $\mathfrak{o}_{J_{G^{\vee}}}$ are fully faithful, we obtain our desired monoidal natural isomorphism.
\end{proof}

\begin{proof}[Proof of \cref{lemma that relevant Coulomb branches are Kostant sections}]
Let \begin{align*}
    \overline{p}_n: &\GL(V_{A_n})/\mathbb{G}_m \to \PGL_{n}&p_n:& \GL(V_{A_n})\to \GL_{n}\\ 
    \overline{q}_{\vec{m}}: &\GL(V_{A_{\vec{m}}})/\mathbb{G}_m \to \overline{L} &q_{\vec{m}}:&\GL(V_{A_{\vec{m}}}) \to L
\end{align*} denote the respective maps given or induced by the obvious projection map. We let $\mathcal{A}_R^{G^{\vee}}$ denote the algebra object given by the image under derived geometric Satake of $\mathcal{O}(T^*(G^{\vee}))$, as discussed in \cite[Section 5(vii)]{BravermanFinkelbergNakajimaRingObjectsIntheEquivariantDerivedSatakeCategoryArisingFromCoulombBranches}. 
Recall that \cite[Theorem 2.11]{BravermanFinkelbergNakajimaRingObjectsIntheEquivariantDerivedSatakeCategoryArisingFromCoulombBranches} gives an isomorphism
\begin{equation}\label{Isomorphisms of Algebra Objects for SLn}\overline{p}_{n*}(\mathcal{A}_{\GL_{A_{n}}/\mathbb{G}_m, \mathbf{N}_{A_n}}) \cong \mathcal{A}_{R}^{\SL_{n}}\end{equation}
of algebra objects in $\sphericalHeckeCatForTHISGROUP{\PGL_n}$.  Letting $q$ denote the composite \[\prod_{i = 1}^l\GL_{Q_{m_i}}/\mathbb{G}_m \xrightarrow{\prod\overline{p}_{m_i}} \prod_{i = 1}^l\PGL_{m_i} =: \PGL_{\vec{m}}\] we may take the exterior product of this isomorphism and obtain an isomorphism \begin{equation}\label{Isomorphisms of Algebra Objects for Product of PGLs}q_*(\mathcal{A}_{\PGL_{\vec{m}}, \mathbf{N}_{A_{\vec{m}}}}) \cong \mathcal{A}_{R}^{\PGL_{\vec{m}}}\end{equation} of algebra objects in $\sphericalHeckeCatForTHISGROUP{\PGL_{\vec{m}}}$. Now observe that if $H \to H'$ is a central isogeny, then the pullback by the induced map $\Gr_H \to \Gr_{H'}$ maps $\mathcal{A}_R^{H'}$ to $\mathcal{A}_R^H$. Therefore, by applying base change of algebra objects \cite[Proposition 3.3]{GannonWilliamsDifferentialOperatorsOnBaseAffineSpaceofSLnandQuantizedCoulombBranches} to the isomorphism \cref{Isomorphisms of Algebra Objects for SLn}, respectively \cref{Isomorphisms of Algebra Objects for Product of PGLs}, it follows that there are isomorphisms \begin{equation}\label{Isomorphisms of Algebra Objects for All But SLn}p_{n*}(\mathcal{A}_{\GL_{A_n} \mathbf{N}_{A_n}}) \cong \mathcal{A}_{R}^{\GL_{n}}\text{, respectively } \overline{q}_{\vec{m}_*}(\mathcal{A}_{\GL_{A_{\vec{m}}}/\mathbb{G}_m,  \mathbf{N}_{A_n}}) \cong \mathcal{A}_R^{\overline{L} } \text{ and }q_{\vec{m}_*}(\mathcal{A}_{\GL_{A_{\vec{m}}},  \mathbf{N}_{A_n}}) \cong \mathcal{A}_R^{L}\end{equation} of algebra objects in $\sphericalHeckeCatForTHISGROUP{\GL_n}$, respectively $\sphericalHeckeCatForTHISGROUP{\overline{L}}$ and $\sphericalHeckeCatForTHISGROUP{L}$. Taking the cohomology of the isomorphisms in \cref{Isomorphisms of Algebra Objects for SLn} and \cref{Isomorphisms of Algebra Objects for All But SLn} we obtain our desired isomorphism. The fact that these isomorphisms are equivariant follows immediately from the compatibility of \cref{Satake On Free Objects with Compatibility}.
\end{proof}

\newcommand{\quotientOfHByK}{G}
\newcommand{\LieAlgOfQuotientOfHByK}{\mathfrak{g}}

\newcommand{\biggroup}{H}
\newcommand{\Lbig}{\mathfrak{h}}
\newcommand{\normalgroup}{K}

\subsection{Recovering the $S$-dual from the Coulomb branch}
\label{section: recovering the S-dual}
In this section, we briefly discuss connections of the above balanced product constrution to \textit{relative Langlands duality}, a mathematical manifestation of 3-d and 4-d mirror symmetry.  Interest in the Coulomb branch in the mathematics literature has been largely driven by its appearance in 3-d mirror symmetry, as a \lq 3-d mirror\rq{} of the \textit{Higgs branch} of a $G$-representation $\bfN$, defined by  \[\M_{H}(G,\bfN) :=T^*\bfN/\!\!/\!\!/\!\!/G := (T^*\bfN\times_{\LGd} \{0\})\sslash G.\]  In other words, the Higgs branch is the hyperk\"ahler reduction of $\mathbf{M} := T^*\bfN$ by $G$.  Physically, 3-d mirror symmetry expresses an equivalence between the 3-d $\mathcal{N}=4$ $\sigma$-models into these spaces--see \cite{bravermanfinkelbergsurvey}, \cite{webster3DimensionalMirrorSymmetry2023}, \cite{HilburnRaskinTatesThesisinDeRhamSetting} for more on this philosophy.  

Relative Langlands duality extends this idea by considering 3-d boundary conditions of 4-d super-Yang-Mills for the group $G$;  physically, we expect each to have a dual 3-d boundary condition of 4-d super-Yang-Mills for the group $G^{\vee}$. Mathematically, this extension was pioneered by Ben-Zvi, Sakellaridis, and Venkatesh \cite{BenZviSakleredisVenkateshRelativeLanglandsDuality}, who argued that in many cases of interest (such as in the hyperspherical case studied in \cite{BenZviSakleredisVenkateshRelativeLanglandsDuality}) the \lq relative Langlands dual\rq{} of a boundary condition that can be written as a $\sigma$-model into some Hamiltonian $G$-scheme $\mathbf{M}$ is a $\sigma$-model into a Hamiltonian $G^{\vee}$-scheme $\mathbf{M}^{\vee}$. Moreover, if $\mathbf{M} = T^*\bfN$ for some smooth affine $G$-variety $\bfN$, the $G^{\vee}$-scheme $\mathbf{M}^{\vee}$ is expected to be affine as well, and to have its ring of functions given by the cohomology of the algebra object corresponding to $\mathcal{A}_{G, \bfN} \in \sphericalHeckeCatForTHISGROUP{G}$ under the derived geometric Satake equivalence; this perspective is used implicitly in \cite{BravermanFinkelbergNakajimaRingObjectsIntheEquivariantDerivedSatakeCategoryArisingFromCoulombBranches} and is stated explicitly in \cite{BenZviSakleredisVenkateshRelativeLanglandsDuality}, \cite{NakajimaSDualOfHamiltonianGSpacesandRelativeLanglandsDuality}, see in particular \cite[Conjecture 8.1.8]{BenZviSakleredisVenkateshRelativeLanglandsDuality} and \cite[Definition (3.6)]{NakajimaSDualOfHamiltonianGSpacesandRelativeLanglandsDuality}. In what follows, we \textit{define} $\mathbf{M}^{\vee}$ through this expectation: in other words, to any finite dimensional $G$-representation $\bfN$, we define its {\bf S-dual} or {\bf relative Langlands dual} as the affine scheme $\mathbf{M}^{\vee}$ whose ring of functions is the cohomology of the algebra object that maps to $\mathcal{A}_{G, \bfN}$ under the derived geometric Satake equivalence; see \cite[Definition (3.6)]{NakajimaSDualOfHamiltonianGSpacesandRelativeLanglandsDuality} for an equivalent formulation of this definition. By construction, $\mathbf{M}^{\vee}$ is an affine $G^{\vee}$-scheme and admits a $G^{\vee}$-equivariant map $\mathbf{M}^{\vee} \to \LGcd$. 


We now pose a conjecture which informally states that the $S$-dual of a $G$-representation is determined by its Coulomb branch. To state this conjecture precisely, we first recall that, as above, we identify $\M(\quotientOfHByK, 0)$ with the sub-group-scheme $J_{\quotientOfHByK^{\vee}}\subset \quotientOfHByK^{\vee}\times \fc{G}$ of the trivial group scheme with fiber given by the Langlands dual $\quotientOfHByK^{\vee}$. From this isomorphism and the discussion of the $\M_C(G, 0)$-action on $\M_C(G, \bfN)$ in \cref{Action for General Coulomb Branches}, we obtain an action of $J_{G^{\vee}}$ on $\M_C(G, \bfN)$. We also observe that if $X$ is an affine scheme over $\fc{\quotientOfHByK}$ with $J_{\quotientOfHByK^{\vee}}$-action, then
\[G^{\vee} \times^{J_{\quotientOfHByK^{\vee}}}X := (\quotientOfHByK^{\vee} \times\fc{\quotientOfHByK})\times_{\fc{\quotientOfHByK}}^{J_{\quotientOfHByK^{\vee}}}X\] naturally acquires a $G^{\vee}$-action and a $G^{\vee}$-equivariant map $\overline{\mu}$ to $\LGcd$. Indeed, the $G^{\vee}$-representation is induced by left multiplication by $G^{\vee}$ on $\quotientOfHByK^{\vee} \times\fc{\quotientOfHByK}$, and the map $\overline{\mu}$ is induced by the map \[ \mu\colon \quotientOfHByK^{\vee} \times X \to \LGc^*\qquad\qquad\mu(q, x)=\operatorname{Ad}_{q}^*\varkappa(x)\] where $\varkappa \colon X\to (\LieAlgOfQuotientOfHByK^{\vee})^*$ is the composition of the structure map to $\LieAlgOfQuotientOfHByK\sslash \quotientOfHByK\cong \LGc^*\sslash \quotientOfHByK^{\vee}$ followed by the inclusion $\iota_{G^{\vee}}$. Note that since the map $\quotientOfHByK^{\vee} \times X\to \quotientOfHByK^{\vee} \times^{J_{\quotientOfHByK^{\vee}}} X$ might not be surjective, the induced map $\overline{\mu}$ may have irregular elements of $\LGc^*$ in its image.

\begin{Conjecture}\label{Recoverability Conjecture}
There is a $G^{\vee}$-equivariant isomorphism $\mathbf{M}^{\vee} \cong G^{\vee} \times^{J_{G^{\vee}}}\M_C(G, \bfN)$ of affine schemes over $\LGcd$.
\end{Conjecture}

After circulating a version of this preprint, the authors learned that this conjecture has also been independently formulated by Nakajima, and a related discussion will appear in forthcoming work of Ben-Zvi--Sakellaridis--Venkatesh \cite{BZSV2}. 
Moreover, a similar conjectural description of $\mathbf{M}^{\vee} $ for $\mathbf{M}=T^*(G/H)$, the cotangent bundle of an affine spherical variety, was given by Devalapurkar in \cite[Remark 3.6.7]{DevalapurkarKU}, building on the results of \cite[Theorem 3.6.4]{DevalapurkarKU};  
\cref{Recoverability Conjecture} and Devalapurkar's conjecture apply to very few of the same choices of $\mathbf{M}$, but can be interpreted as philosophically the same.  

\subsubsection{Extension of \cref{Recoverability Conjecture}}We now give an informal discussion of an extension of \cref{Recoverability Conjecture}, guided by the heuristic that, in the language of relative Langlands duality, 3-d mirror symmetry is ``relative Langlands duality for the trivial group.''  From this heuristic, one might expect a 3d analogue \begin{equation}\label{eq:3d-mirror} 
    [\{e\} \curvearrowright \M_{H}(G,\bfN)]^{\vee}=[\{e\} \curvearrowright \M_{C}(G,\bfN)]
\end{equation}  of the duality discussed in \cite{NakajimaSDualOfHamiltonianGSpacesandRelativeLanglandsDuality}. Note that we are typically in a situation where one cannot use \cite[Definition (3.6)]{NakajimaSDualOfHamiltonianGSpacesandRelativeLanglandsDuality}, since $\M_{H}(G,\bfN)$ is typically not a vector space, and we discuss below in \cref{rem: grain-of-salt}, even if it is, some care is required to translate \cref{eq:3d-mirror} into correct mathematically precise statements.  The identification \cref{eq:3d-mirror} results from \cite[Definition (3.6)]{NakajimaSDualOfHamiltonianGSpacesandRelativeLanglandsDuality} by reduction as in \cite[Theorem 4.2]{NakajimaSDualOfHamiltonianGSpacesandRelativeLanglandsDuality} using the identification (see, for example, \cite[Example 3.8]{NakajimaSDualOfHamiltonianGSpacesandRelativeLanglandsDuality})
\[[pt \curvearrowleft G]^{\vee} =[\fc{G}\times G^{\vee}\curvearrowleft G^{\vee}].\]
This is a reinterpretation of gauging both sides of the equation by $G$, as discussed in  \cite[\S 5(xi)]{BravermanFinkelbergNakajimaRingObjectsIntheEquivariantDerivedSatakeCategoryArisingFromCoulombBranches}.


\begin{Remark}\label{rem: grain-of-salt}
    How to interpret the heuristic suggested in \cref{eq:3d-mirror} as a precise mathematical statement remains a tricky question.  There are many examples of $G,\bfN$ where the same affine variety appears as $\M_{H}(G,\bfN)\cong \M_{H}(G',\bfN')$ but $\M_{C}(G,\bfN)$ and $ \M_{C}(G',\bfN')$ have different dimensions (and thus are not isomorphic as affine varieties).  For example, if $\bfN=\C^n,G=\Gm^n$, then $\M_{H}(G,\bfN)=0$, but $\M_{C}(G,\bfN)=\mathbb{A}^{2n}$.  It seems likely that we need to consider additional structure on both sides of the duality to make this into a true bijection.  
\end{Remark}

We can naturally interpolate between \cref{Recoverability Conjecture} and \cref{eq:3d-mirror} by considering a partial gauging.  That is, 
let $\biggroup\rhd\normalgroup$ be a group and a normal subgroup such that $\quotientOfHByK=\biggroup/\normalgroup$.  
Let $\mathbf{M}=T^*\bfN$ for $\bfN$ a finite-dimensional $\biggroup$-representation and as above $\M_H(K,\bfN)=\mu_K^{-1}(0)\sslash\normalgroup$ be the complex symplectic quotient by the subgroup $K$, where $\mu_K$ is the moment map for the $K$-action.  The quotient $\M_H(K,\bfN)$ is a Hamiltonian $\quotientOfHByK$-space with induced moment map. 
With this notation, we can extend \cref{Recoverability Conjecture} to suggest that:
\begin{equation}\label{eq:more general conjecture}
    [G\curvearrowright\M_H(K,\bfN) ]^{\vee} \cong [G^{\vee} \curvearrowright G^{\vee} \times^{J_{\quotientOfHByK^{\vee}}}\M_C(\biggroup, \bfN)]
\end{equation}
This includes \cref{eq:3d-mirror} as the special case $G=K$ and \cref{Recoverability Conjecture} as the special case $K=\{1\}$.  
    This statement should be taken with the same grain of salt as discussed in Remark \ref{rem: grain-of-salt}.  Making it precise will likely require considering additional structure on $\M_H(K,\bfN)$ which at present is implicit in how it is written as a quotient.  

\begin{Remark}\label{rem:Gaiotto-Witten}
    One can view \cref{eq:more general conjecture} as a mathematical interpretation of one of the constructions of \cite[\S 4.3]{gaiottoSDuality2009}.  When translating notation, note that our $G^{\vee}$ is Gaiotto and Witten's $G$ (and we only consider the case where Gaiotto and Witten's $H$ is equal to $G$).  Thus, in our notation, they say that the mirror $[G\curvearrowright\M_H(K,\bfN) ]^{\vee}$ (thought of as coupling with the 3-d theory $\mathcal{T}=\operatorname{Hyp}(\mathbf{M})\# K$, which is $\mathfrak{B}^{\vee}$ in the notation of \cite[\S 4.2]{gaiottoSDuality2009}) can be described via an operation Gaiotto and Witten call \lq\lq composite\rq\rq, where we take the 3-d theory $\mathcal{T}\times T[G]$, gauged by the group $G$ acting diagonally on the two theories, and then take the Coulomb branch of that theory (that is, we take the 3-d mirror of the composite theory).   We can see this more clearly by gauging by $G\times G$ first.  The result is the theory $\operatorname{Hyp}(\mathbf{M})\times T[G]$, gauged by the group $H\times G$.  The Coulomb branch of this theory is $\M_C(H,\bfN)\times (\fc{G}\times G^{\vee})$ and we recover the desired composite by restricting to the diagonal subgroup $H\subset H\times G$.  Of course, \cref{Restriction of Groups is Hamiltonian Reduction} suggests the result of this restriction should be \cref{eq:more general conjecture}.  
\end{Remark}

\printbibliography
\end{document}